\documentclass[10pt,twoside,a4paper]{article}

\def\ShowComments{}

\usepackage[utf8]{inputenc}

\usepackage[T1]{fontenc}

\usepackage{amsmath}
\numberwithin{equation}{section}

\usepackage{amssymb}

\usepackage{microtype}

\usepackage[backend=biber,maxnames=10,backref=true,hyperref=true,safeinputenc]{biblatex}
\DefineBibliographyStrings{english}{%
	backrefpage = {cited on page},
	backrefpages = {cited on pages},
}

\DeclareFieldFormat[report]{title}{``#1''}
\DeclareFieldFormat[book]{title}{``#1''}
\AtEveryBibitem{\clearfield{url}}
\AtEveryBibitem{\clearfield{note}}

\usepackage{graphicx}
\graphicspath{{images/}}

\usepackage{tikz}
\usepackage{pgfplots}
\pgfplotsset{compat=newest}

\title{Regularization with Metric Double Integrals for Vector Tomography}

\author{Melanie Melching$^1$ \\{\footnotesize\href{mailto:melanie.melching@univie.ac.at}{melanie.melching@univie.ac.at}}
	\and Otmar Scherzer$^{1,2}$\\{\footnotesize\href{mailto:otmar.scherzer@univie.ac.at}{otmar.scherzer@univie.ac.at}}}
\date{\today}

\usepackage[english]{babel}
\usepackage{mathtools}
\usepackage[labelformat=simple]{subcaption}

\newcommand{\sphere}{\mathbb{S}^1}

\newcommand{\F}{\mathcal{F}}
\newcommand{\Reg}{\mathrm{\Phi}}
\newcommand{\op}{\mathrm{F}}
\newcommand{\radon}{\mathcal{R}}
\newcommand{\ray}{\mathcal{D}}

\newcommand{\dx}{\,\mathrm{d}x}
\newcommand{\dxy}{\,\mathrm{d}(x,y)}
\newcommand{\drp}{\,\mathrm{d}(r,\varphi)}
\newcommand{\dH}{\ \mathrm{d}\sigma}
\newcommand{\defeq}{\vcentcolon=}
\newcommand{\eqdef}{=\vcentcolon}

\newcommand{\dRM}{d_{\R^m}}
\newcommand{\dS}{d_{\sphere}}

\newcommand{\abs}[1]{\left|#1\right|}
\newcommand{\norm}[1]{\left\|#1\right\|}
\newcommand{\enorm}[1]{\left|#1\right|}
\newcommand{\set}[1]{ \left\{ #1 \right\}}

\newcommand{\level}[2]{\text{level}(#1;#2)}

\newcommand{\vd}{{v^\delta}}

\usepackage[pdftex,colorlinks=true,linkcolor=blue,citecolor=green,urlcolor=blue,bookmarks=true,bookmarksnumbered=true]{hyperref}
\hypersetup
{
	pdfauthor={M. Melching, O. Scherzer},
	pdfsubject={Subject},
	pdftitle={Regularization with Metric Double Integrals for Vector Tomography},
	pdfkeywords={regularization, vector-valued data, non-convex, metric, double integral, fractional Sobolev space, tomography}
}

\usepackage[hyperref,amsmath,thmmarks]{ntheorem}
\usepackage{aliascnt}

\newtheorem{lemma}{Lemma}[section]

\newaliascnt{proposition}{lemma}

\aliascntresetthe{proposition}

\newaliascnt{example}{lemma}

\aliascntresetthe{example}

\newaliascnt{remark}{lemma}
\newtheorem{remark}[remark]{Remark}
\aliascntresetthe{remark}

\newaliascnt{notation}{lemma}

\aliascntresetthe{notation}

\newaliascnt{corollary}{lemma}
\newtheorem{corollary}[corollary]{Corollary}
\aliascntresetthe{corollary}

\newaliascnt{theorem}{lemma}
\newtheorem{theorem}[theorem]{Theorem}
\aliascntresetthe{theorem}

\theorembodyfont{\normalfont}
\newaliascnt{definition}{lemma}
\newtheorem{definition}[definition]{Definition}
\aliascntresetthe{definition}

\newaliascnt{assumption}{lemma}
\newtheorem{assumption}[assumption]{Assumption}
\aliascntresetthe{assumption}

\theoremstyle{nonumberplain}
\theoremseparator{:}
\theoremheaderfont{\normalfont\itshape}

\theoremsymbol{\ensuremath{\square}}
\newtheorem{proof}{Proof}

\usepackage[a4paper,centering,bindingoffset=0cm,marginpar=2cm,margin=2.5cm]{geometry}

\usepackage[pagestyles]{titlesec}
\titleformat{\section}[block]{\large\sc\filcenter}{\thesection.}{0.5ex}{}[]
\titleformat{\subsection}[runin]{\bf}{\thesubsection.}{0.5ex}{}[.]

\newpagestyle{headers}%
{%
	\headrule
	\sethead%
	[\footnotesize\thepage]%
	[\footnotesize\sc M. Melching, O. Scherzer]%
	[]%
	{}%
	{\footnotesize\sc Regularization with Metric Double Integrals for Vector Tomography}%
	{\footnotesize\thepage}
	\setfoot{}{}{}
}
\usepackage[font=footnotesize,format=plain,labelfont=sc,textfont=sl,width=0.75\textwidth,labelsep=period]{caption}
\usepackage{dsfont}
\newcommand{\N}{\mathbb{N}}

\newcommand{\R}{\mathbb{R}}

\newcommand{\e}{\mathrm e}
\let\ii\i
\renewcommand{\i}{\mathrm i}
\renewcommand{\d}{d}

\ifdefined\ShowComments{}
  \newcommand{\commentO}[1]{\textcolor{red}{#1}}
  \newcommand{\commentM}[1]{\textcolor{orange}{#1}}
  
\else
  \newcommand{\commentO}[1]{}
  \newcommand{\commentM}[1]{}
\fi

\begin{document}
\renewcommand{\sectionautorefname}{Section}
\renewcommand{\subsectionautorefname}{Subsection}
\maketitle
\thispagestyle{empty}
\begin{center}
\hspace*{5em}
\parbox[t]{12em}{\footnotesize
\hspace*{-1ex}$^1$Computational Science Center\\
University of Vienna\\
Oskar-Morgenstern-Platz 1\\
A-1090 Vienna, Austria}
\hfil
\parbox[t]{17em}{\footnotesize
\hspace*{-1ex}$^2$Johann Radon Institute for Computational\\
\hspace*{1em}and Applied Mathematics (RICAM)\\
Altenbergerstraße 69\\
A-4040 Linz, Austria}
\end{center}

  \ifdefined\ShowTableOfContents{}  \tableofcontents  \fi
  
\begin{abstract}
	We present a family of \emph{non-local} variational regularization methods for solving \emph{tomographic} problems,   
	where the solutions are functions with range in a closed subset of the Euclidean space, for example if the 
	solution only attains values in an embedded sub-manifold. Recently, in \cite{CiaMelSch19}, such regularization methods 
	have been investigated analytically and their efficiency has been tested for basic imaging tasks such as denoising and 
	inpainting. In this paper we investigate solving complex vector tomography problems with non-local variational methods both analytically and numerically.
\end{abstract}

\section{Introduction}
In this paper we study the stable solution of \emph{tomographic} imaging problems with a \emph{derivative free} variational 
regularization technique, recently introduced in \cite{CiaMelSch19}, which takes into account a-priori information that the 
function values of the solution are contained in some subset $K$ of an Euclidean space $\R^m$. Particular applications are 
tomographic reconstructions of $2D$ flow fields from acoustic time-of-flight measurements (see for instance \cite{BraHau91,Nor97}), 
in $3D$ this problem was considered in \cite{Juh92,SpaStrLinPer95}. Similar tomographic imaging problems also appear in the context of 
Doppler ultrasound imaging. Opposed to \emph{standard} tomographic imaging, consisting in inverting the Radon and the ray transform for 
\emph{intensity} valued functions (see for example \citeauthor{Nat01} \cite{Nat01} and for an overview on applications see \cite{Dea83}),
vector field tomography is much less advanced (see \cite{Sch08}). 

In contrast to \cite{CiaMelSch19}, where simple image analysis tasks, such as denoising and inpainting have been 
investigated, the focus of this paper is on solving \emph{vector tomography} problems which involve the Radon $\radon$ or the ray 
transform $\ray$, respectively. \emph{Nonlinear} imaging tasks, such as \emph{registration} 
(see for instance \cite{BauBruMic14,DroRum04, IglRumSch17, MilYou01,Mod03,PoeModSch10}, to name but a few) and 
\emph{tomographic displacement estimations} \cite{NeuPerSte18b}, fit in the framework of this paper, but are not considered here. 

All along this paper $\op$ is a subsumption for the Radon transform $\radon$ and the ray transform $\ray$, 
such that the considered tomographic imaging problem can be written in a unified manner as the operator equation 
\begin{equation} \label{eq:op}
	\op[w] = v^0
\end{equation}
on a set of functions with range in a closed subset $K \subseteq \R^m$. That is 
$$w:\Omega \subset \R^n \to K \text{ and } v^0:\Sigma := \set{ (r,\varphi):r \geq 0, \varphi \in \mathbb{S}^{n-1} } \to \R^M,$$
denotes vector-valued sinogram data, where the parameter $M$ is introduced in the following way to be able to perform a uniform analysis 
for both kind of vector tomographic problems:
\begin{align}\label{eq:M}
    M = m \text{ if } \op = \radon \text{ and }
    M = 1 \text{ if } \op = \ray.
\end{align} 
$K$ is associated with a metric $\d$, which determines an appropriate distance measure for elements of $K$, which is not necessarily 
the Euclidean distance.

We assume that noisy measurement data $\vd$ of $v^0$ are available. We avoid direct solution of \autoref{eq:op} 
and implement variational regularization methods to deal with numerical instabilities. The method of choice consists in approximating 
the solution of \autoref{eq:op} by some minimizer of the \emph{metric double integral regularization} functional with some appropriately 
chosen parameter $\alpha$, an indicator $l \in \set{ 0,1 }$ and an exponent $p \in (1,+\infty)$:
\begin{equation} \label{eq:reg}\boxed{
	\begin{aligned} 
		\F_{[d]}^{\alpha,\vd}(w) \defeq& \int_\Sigma \enorm{\op[w](r,\varphi)- \vd(r,\varphi)}^p \drp + 
		\alpha \Reg_{[\d]}^l (w), \\
		\text{ with } \Reg^l_{[\d]}(w) \defeq & 
		\int_{\Omega\times \Omega} \frac{\d^p(w(x), w(y))}{\enorm{x-y}^{n+p s}} \rho^l(x-y) \dxy.
	\end{aligned}}
\end{equation}
The particular choice of the regularization functional is motivated by the following properties of $\Reg^l_{[\dRM]}$, that is, when $\d = \dRM$ is the 
Euclidean metric. 
\begin{enumerate}
 \item For $l=0$, $p \in (1,+\infty)$ and $s \in (0,1)$ 
       \begin{equation*}
		\Phi_{[\dRM]}^0(w) \defeq \int_{\Omega\times \Omega} \frac{\enorm{w(x)-w(y)}^p}{\enorm{x-y}^{n+p s}} \dxy
       \end{equation*}
       is the \emph{fractional Sobolev semi-norm} of order $s$ to the power $p$ (see for instance \cite{Wlo82}). 
 \item Compared to the fractional Sobolev semi-norm the additional function $\rho$ in $\Reg_{[\dRM]}^1$ is beneficial 
       for numerical implementation. We choose a function $\rho$ which is rotationally symmetric around the origin, strictly 
       positive, and decays rapidly from the origin. In this way it concentrates the evaluation of the double integral to the 
       central diagonal of $\Omega \times \Omega$, while still guaranteeing that $\Reg_{[\dRM]}^1$ is an equivalent fractional Sobolev 
       semi-norm (see \autoref{le:Equiv}).
 \item For $\d = \d_\R$ and a family $(\rho_\varepsilon)_{\varepsilon > 0}$ of non-negative, 
       radially symmetric, radially decreasing mollifiers, it is shown in \cite{BouBreMir01,Pon04b} that the following relation holds
       \begin{equation} \label{eq:double_integral}
        \begin{aligned}
		\lim_{\varepsilon \searrow 0} \tilde{\Reg}_\varepsilon^1(w) & 
		\defeq \lim_{\varepsilon \searrow 0}  \int_{\Omega\times \Omega} \frac{\enorm{w(x)- w(y)}^p}{\enorm{x-y}^p} \rho_\varepsilon(x-y) \dxy\\
		&= 
		\begin{cases}
			K(p,n)\abs{w}^p_{W^{1,p}} & \mbox{if } w \in W^{1,p}(\Omega;\R), 1 < p < \infty, \\
			\infty & \mbox{otherwise}.
		\end{cases}
	\end{aligned}
\end{equation}
This relation, in particular, shows that the regularization functionals $\Reg_{[\d]}^1$ can be considered as approximate Sobolev semi-norms
of set-valued functions.

$\tilde{\Reg}_\varepsilon^1$ from \autoref{eq:double_integral} has been implemented as a regularizer in standard tomographic 
image reconstructions in \cite{AubKor09,BerTre10}. 
\end{enumerate}

The conceptual advantage of $\Reg_{[\d]}^1$, in contrast to standard Sobolev and total-variation minimization, is that the \emph{natural} 
metric of $K$ can be included in the regularization functional. We bring some examples in \autoref{ss:k}.


This paper is organized as follows: In \autoref{sec:2} we review regularization results from \cite{CiaMelSch19}. In \autoref{sec:3} the Radon 
and ray transform, respectively, are recalled and we verify the general conditions of \cite{CiaMelSch19} 
showing that minimization of the functional $\F_{[d]}^{\alpha,\vd}$ is well--posed, stable and convergent, in the sense of a regularization 
method, also for tomographic imaging. The main objectives of the paper are to provide case studies of using $\F_{[d]}^{\alpha,\vd}$, as 
defined in \autoref{eq:reg}, for tomographic reconstructions of vector fields. Particular emphasis is devoted to analyze the effect of the 
metric $\d$ in numerical experiments (see \autoref{sec:5}). Typical differences can be observed for instance in
\autoref{fig:fig3} - \autoref{fig:fig5}, below.

\section{Regularization theory for Sobolev functions with values in a closed set}
\label{sec:2}

We start this section by making the following assumptions which are valid all along this paper: 
\begin{assumption} \label{ass:gen} 
	\mbox{}                         
	\begin{enumerate}
		\item \label{itm:setting}
		$p\in (1, +\infty)$, $s \in (0,1)$, $l \in \{0,1\}$, $\Omega \subset \R^n$ is a  nonempty, bounded and 
		connected open set with Lipschitz boundary which is compactly supported in a ball of radius $R$, $\mathcal{B}_R(0)$, 
		and $K \subseteq \R^m$ is a nonempty and closed subset of $\R^m$. 
		By $\enorm{ \cdot }$ we denote the Euclidean norm on $\R^n$ and $\R^m$, respectively.
		\item \label{itm: assmetrik}
		$\d :K \times K \rightarrow [0, +\infty)$ denotes a metric on $K$
		which is \emph{normalized equivalent} to the Euclidean distance restricted to $K \times K$, that is 
        \begin{equation}\label{eq:MI}
	      \enorm{ a-b } \leq \d(a,b) \leq C_u\enorm{a-b}, \quad a,b \in K, C_u > 0.
        \end{equation}
        Note, that the terminology normalized equivalent refers to the assumption that the lower equivalence bound in \autoref{eq:MI} is equal to one.
	\end{enumerate}
\end{assumption}

We continue by recalling definitions of spaces of functions attaining values in a closed subset $K \subseteq \R^m$, which is 
associated with a metric $d$, and we summarize some notation, which is used throughout the paper. 

First we introduce the notion of mollifiers:
\begin{definition}[Mollifier]\label{de:mol} 
 We call a non-negative, radially symmetric function $\rho \in C_c^\infty(\R^n;\R)$ satisfying $\int_{\R^n} \rho (x) \dx = 1$ 
 a \emph{mollifier}. We say that a mollifier satisfies the \emph{separation property} if for all $0 < \tau < \norm{\rho}_{L^\infty(\R^n;\R)}$ 
 there exists $\eta > 0$ such that
 \begin{equation} \label{eq:sep}
    \set{ z \in \R^n :\rho(z) \geq \tau } = \set{z \in \R^n : \enorm{z} \leq \eta }.
 \end{equation}
 This condition holds for instance if $\rho$ is a radially decreasing mollifier satisfying $\rho(0) > 0$.
\end{definition}

\begin{figure}[h!]
	\centering
	\begin{tikzpicture}
	\draw[->] (-2,0) -- (2,0);
	\node[right] at (2,0) {$x$};
	\draw[->](0,0) -- (0,3);
	\node[above] at (0,3) {$\rho(x)$};
	\draw[ thick, , out=0, in=180] (0,2.5) to (1.5,0);
	\draw[ thick, , out=180, in=0] (0,2.5) to (-1.5,0);
	\draw[thick, ] (1.5,0) -- (2,0);
	\draw[thick, ] (-1.5,0) -- (-2,0);	
	\draw[dashed, ] (-0.6,2) -- (0.6,2);		
	\node[right, ] at (0.7,2) {$\tau$};
	\draw[thick, ] (-0.6,0) -- (0.6,0);		
	\draw[dashed] (-0.6,0) to (-0.6,2);
	\draw[dashed] (0.6,0) to (0.6,2);
	\draw [thick,
	decoration={brace, mirror, raise=0.4cm},
	decorate
	] (-0.6,0) -- (0.6,0);
	\node[below] at (0,0){$0$};
	\node[below, ] at (0,-0.45){$2\eta$};
	
	\draw[dashed, ] (-0.6,2.5) -- (0.6,2.5);
	\node[right, ] at (0.7,2.5) {$\norm{\rho}_{L^\infty(\R;\R)}$};
	\end{tikzpicture}
	\caption{Example of a mollifier fulfilling the separation property in the case $n=1$.}
\end{figure}
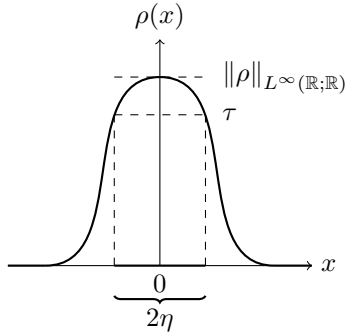
	
In the following we define sets of Sobolev functions and associated semi-norms:
\begin{definition}[Sobolev spaces of fractional order]\label{def:Spaces} Let \autoref{ass:gen} hold.
	\begin{itemize}
		\item We denote by $L^p(\Omega;\R^m)$ the \emph{Lebesgue space} of vector-valued functions.
		\item The Sobolev space $W^{1,p}(\Omega;\R^m)$ consists of all weakly differentiable functions in $L^p(\Omega;\R^m)$ for which
		\begin{equation*}
		\norm{w}_{W^{1,p}(\Omega;\R^m)} 
		\defeq \left( \norm{w}_{L^p(\Omega;\R^m)}^p +  \int_\Omega \enorm{\nabla w(x)}^p \dx \right)^{1/p} < \infty\;,
		\end{equation*}
		where $\nabla w$ is the Jacobian of $w$ and $\abs{w}_{W^{1,p}(\Omega;\R^m)}^p \defeq \int_\Omega \enorm{\nabla w(x)}^p \dx$ is the Sobolev semi-norm.
		\item The \emph{fractional Sobolev space} of order $s$ is defined (cf. \cite{Ada75}) as the set
		\begin{gather*}
		W^{s,p}(\Omega;\R^m) \defeq  \set{ w \in L^p(\Omega;\R^m) :\frac{\enorm{w(x)-w(y)}}{\enorm{x-y}^{\frac{n}{p}+s}} \in L^p (\Omega \times \Omega;\R) } 
		\end{gather*}
		equipped with the norm
		\begin{equation}\label{eq:s_norm}
		\norm{w}_{W^{s,p}(\Omega;\R^m)} \defeq \left(\norm{w}_{L^p(\Omega;\R^m)}^p + \abs{w}_{W^{s,p}(\Omega;\R^m)}^p \right)^{1/p},
		\end{equation}
		where $\abs{w}_{W^{s,p}(\Omega;\R^m)}$ is the semi-norm on $W^{s,p}(\Omega;\R^m)$, defined by
		\begin{equation}\label{eq:s_seminorm}
		\abs{w}_{W^{s,p}(\Omega;\R^m)} \defeq  
		\left(\int_{\Omega\times \Omega} \frac{\enorm{w(x)-w(y)}^p}{\enorm{x-y}^{n+ps}} \dxy \right)^{1/p}
		\quad \text{ for all } w \in W^{s,p}(\Omega;\R^m).
		\end{equation}
		\item We define the \emph{fractional Sobolev set} of order $s$ with data in $K$ as
        \begin{equation*}
            W^{s,p}(\Omega;K) \defeq \set{ w \in W^{s,p}(\Omega;\R^m) : w(x) \in K \text{ for a.e. } x \in \Omega }.
        \end{equation*}
        The \emph{Lebesgue set} with data in $K$ is defined as
			\begin{equation*}
				L^p(\Omega;K)  \defeq \set{ w \in L^p(\Omega;\R^m) : w(x) \in K \text{ for a.e. } x \in \Omega}.
			\end{equation*}
\end{itemize}
\end{definition}
Note that $L^p(\Omega;K)$ and $W^{s,p}(\Omega;K)$ are sets and {\bf not} linear spaces because summation of elements in $K$ is typically not closed in $K$.

The proofs of the following lemma can be found in \cite{CiaMelSch19}, which are adaptations of embedding theorems  from \cite{Ada75}
for Sobolev spaces of intensity functions (that is mapping to $\R$). 

\begin{lemma} \label{lem:inclusion} Let \autoref{ass:gen} be satisfied.
	Then
	
	$\bullet$ \emph{ Compact embedding:}
	$W^{s,p}(\Omega;\R^m) \subseteq L^p(\Omega;\R^m)$ and the embedding is compact, meaning that every bounded 
	sequence in $W^{s,p}(\Omega;\R^m)$ has a convergent subsequence in $L^p(\Omega;\R^m)$. Moreover, the space $W^{s,p}(\Omega;\R^m)$ is reflexive.
	
	$\bullet$\emph{ Sequential closedness of {$W^{s,p}(\Omega;K)$} and {$L^p(\Omega;K)$}:} 
	Let $w_* \in W^{s,p}(\Omega;\R^m)$ and $(w_k)_{k\in \N}$ be a sequence in $W^{s,p}(\Omega;K) \subseteq W^{s,p}(\Omega;\R^m)$ with $w_k \rightharpoonup w_*$ weakly in $W^{s,p}(\Omega;\R^m)$. 
	Then, $w_* \in W^{s,p}(\Omega;K)$ and $w_k \rightarrow w_*$ (strongly) in $L^p(\Omega;K) \subseteq L^p(\Omega;\R^m)$. Moreover, there exists a subsequence 
	$(w_{k_j})_{j \in \N}$ converging to $w_*$ pointwise almost everywhere, that is $w_{k_j}(x) \to w_*(x)$ as $j \to \infty$ 
	for almost every $x \in \Omega$.
\end{lemma}

The regularizer $\Reg_{[\dRM]}^l$ is \emph{p-homogeneous} and \emph{sub-additive}. The proof of this is analogous to the 
proof of Minkowski's inequality.
\begin{remark}\label{re:RP} Let \autoref{ass:gen} hold, then for all $v,w \in W^{s,p}(\Omega;\R^m)$ and 
	$\lambda \in \R^m$ 
	\begin{equation}\label{eq:RegProperties}
	\Reg_{[\dRM]}^l(\lambda w) = \enorm{\lambda}^p \Reg_{[\dRM]}^l(w)
	\quad \text{ and } \quad
	(\Reg_{[\dRM]}^l(v + w))^{\frac{1}{p}} \leq (\Reg_{[\dRM]}^l(v))^{\frac{1}{p}} + (\Reg_{[\dRM]}^l(w))^{\frac{1}{p}}.
	\end{equation}
	In the second relation equality only holds if $u$ and $v$ are linearly dependent, equal up to a constant, $v=0$ or $w=0$, respectively. 
\end{remark} 

In the following we state a Poincare type inequality for the regularization functional from \autoref{eq:reg}: 
\begin{lemma}\label{le:PoL}
	Let \autoref{ass:gen} hold. Then there exists a constant $C_P > 0 $ such that for all $w \in W^{s,p}(\Omega;K) \subseteq W^{s,p}(\Omega;\R^m)$ 
	the following holds:
	\begin{equation} \label{eq:PT}
	\norm{w-\overline{w}}_{L^p(\Omega;\R^m)}^p \leq C_P \Reg_{[\d]}^l(w) \text{ where } 
	\overline{w} \defeq \frac{1}{|\Omega|} \begin{pmatrix} 
                                            \int_\Omega w_1(x) \dx\\
                                            \vdots \\
                                            \int_\Omega w_m(x) \dx
	                                       \end{pmatrix}
    \end{equation}
	is the component-wise mean average.
\end{lemma}
In the case $\d=\dRM$ is the Euclidean metric the proof of the lemma is analogous to the proof of \cite[Lemma 3.8]{CiaMelSch19}, which in turn 
is based on the ideas of the proof of Poincaré's inequality in \cite{Eva10}. For general $\d$ we use \autoref{eq:MI} which implies $\Reg_{[\dRM]} \leq \Reg_{[\d]}$.

The subsequent lemma shows that indeed the fractional Sobolev norm (defined in \autoref{eq:s_norm}) and 
{$\Reg_{[\d]}^1(\cdot) + \norm{\cdot}^p_{L^p(\Omega;\R^m)}$, with $\Reg_{[\d]}^1(\cdot)$ defined in \autoref{eq:reg}, are equivalent. 
This will be used later on to prove coercivity of the functional $\F_{[d]}^{\alpha,\vd}$. The statement of the result is rather trivial, when 
the mollifier $\rho$ is uniformly positive (that is it satisfies $0 < \rho_1 \leq \rho(x) \leq \rho_2 < \infty$ for $x \in \Omega$), 
however it also holds when the mollifier has compact support in $\Omega$. 

\begin{lemma}\label{le:Equiv} Let \autoref{ass:gen} hold and assume that the index is $l=1$. Moreover, assume that $\rho$ is a 
    mollifier, which satisfies the separation property \autoref{eq:sep}.
	Then, there exist constants $0 < \underline{C} \leq \overline{C}$ such that for all $w \in W^{s,p}(\Omega;K) \subseteq W^{s,p}(\Omega;\R^m)$
	\begin{equation}\label{eq: Equivalence}
	\underline{C} \left( \norm{w}^p_{L^p(\Omega;\R^m)} + \Reg_{[\d]}^1(w)  \right) 
	\leq \norm{w}^p_{W^{s,p}(\Omega;\R^m)} \leq \overline{C} \left( \norm{w}^p_{L^p(\Omega;\R^m)} + \Reg_{[\d]}^1(w) \right).
	\end{equation}
\end{lemma}

\begin{proof}
	We prove the result for $\d = \dRM$; for general $\d$ the statement follows using \autoref{eq:MI} and using instead of the constant 
	$\overline{C}$ the constant $\overline{C} \max\{1,C_u\}$.
	To show the upper bound in \autoref{eq: Equivalence} we use the same method as in \cite[Lemma 3.7 and Lemma 3.8]{CiaMelSch19} and split $\Omega \times \Omega$ into a set $\mathcal{S}$ consisting of the points close to the central diagonal of $\Omega \times \Omega$ and its complement $\mathcal{S}^c$. \\
	The separation property of $\rho$, \autoref{eq:sep}, ensures that for every $0 < \tau < \norm{\rho}_{L^\infty(\R^n;\R)}$ there exists $\eta  > 0$ such that 
	\begin{equation*} \begin{aligned}
	\mathcal{S} \defeq \set{ (x,y) \in \Omega \times \Omega :\rho(x-y) \geq \tau} 
	= \set{ (x,y) \in \Omega \times \Omega :\enorm{x-y} \leq \eta  }. 
	\end{aligned} \end{equation*}
	We obtain from Jensen's inequality that for all $w \in W^{s,p}(\Omega;K) \subseteq W^{s,p}(\Omega;\R^m)$ 
	\begin{equation*} \begin{aligned}
	\abs{w}^p_{W^{s,p}(\Omega;\R^m)} 
	& = \int_{\mathcal{S}} \frac{\enorm{w(x)-w(y)}^p}{\enorm{x-y}^{n+p s}} \dxy
	+ \int_{\mathcal{S}^c} \frac{\enorm{w(x)-w(y)}^p}{\enorm{x-y}^{n+p s}} \dxy \\
	& \leq \frac{1}{\tau} \int_{\mathcal{S}} \frac{\enorm{w(x)-w(y)}^p}{\enorm{x-y}^{n+p s}} \rho(x-y) \dxy
	+ \frac{2^p \abs{\Omega}}{\eta^{n+p s}}\norm{w}^p_{L^p(\Omega;\R^m)} \\
	& \leq \frac{1}{\tau}\Reg_{[\dRM]}^1(w) + \frac{2^p \abs{\Omega} }{\eta^{n+p s}}\norm{w}^p_{L^p(\Omega;\R^m)}.
	\end{aligned} \end{equation*}
	With $\overline{C} \defeq \max \set{ \frac{1}{\tau}, 1+ \frac{2^p \abs{\Omega} }{\eta^{n+p s}} }$ we get the upper inequality of 
	\autoref{eq: Equivalence}. 
	To prove the lower bound note that since the mollifier $\rho$ satisfies \autoref{de:mol} and $\overline{\Omega}$ is bounded, 
	there exists a constant $\rho_2 > 0$ such that $\norm{\rho}_{L^\infty(\R^n;\R)} \leq \rho_2$.
	
	Then we calculate that
	\begin{equation*}
	\norm{w}_{W^{s,p}(\Omega;\R^m)}^p \geq \norm{w}^p_{L^p(\Omega;\R^m)} + \frac{1}{\rho_2}\Reg_{[\dRM]}^1(w) \geq \min  \set{  1, \frac{1}{\rho_2}  }\left( \norm{w}^p_{L^p(\Omega;\R^m)} + \Reg_{[\dRM]}^1(w) \right).
	\end{equation*}
	Defining $\underline{C} \defeq \min  \set{1, \frac{1}{\rho_2} }$ finishes the proof.
\end{proof}

The choice of the metric $\d$ influences properties of the regularizer, as for instance invariance properties.
\begin{remark}[Invariances]
In the numerical examples in \autoref{sec:5} below we consider the regularizer $\Reg_{[\dS]}^1$ and compare it with the (vectorial) Sobolev 
semi-norm regularizer $\Theta$, defined by,
\begin{equation*}
w \in W^{1,p}(\Omega;\R^m) \to \Theta(w) \defeq \int_\Omega \enorm{\nabla w(x)}_F^p \dx, 
\end{equation*}
where $\enorm{\nabla w(x)}_F$ denotes the Frobenius-norm of the matrix $\nabla w(x)$.

$\Reg_{[\dS]}^1$ is rotation invariant, that is, 
\begin{equation*}
\Reg_{[\dS]}^1(\mathcal{O} w) = \Reg_{[\dS]}^1(w)  \text{ for all } \mathcal{O} = 
\begin{pmatrix}
\cos(\gamma) & -\sin(\gamma) \\
\sin(\gamma) & \cos(\gamma)
\end{pmatrix}, w \in W^{s,p}(\Omega, \sphere) \subset W^{1,p}(\Omega, \R^2),
\end{equation*}
while $\Theta(w)$ is shift invariant: That is, for all $c \in \R^2$,
\begin{equation*}
\Theta(w+c) = \Theta(w) \text{ for all } w \in W^{1,p}(\Omega;\R^2).
\end{equation*}
Both regularizers are reflection invariant, that is
\begin{equation*}
\Reg_{[\dS]}^1(- w) = \Reg_{[\dS]}^1(w), \quad \Theta(-w) = \Theta(w) \text{ for all } w \in W^{1,p}(\Omega;\sphere) \subset W^{1,p}(\Omega, \R^2). 
\end{equation*}
\end{remark}

The next paragraph reviews regularization theory with double integral regularization 
functionals for functions with values in $K$ from \cite{CiaMelSch19}. Their analysis in turn 
is based on the regularization analysis of \cite{SchGraGroHalLen09,SchuKalHofKaz12}. Below, we show 
that the Radon $\radon$ and the ray transform $\ray$, respectively, satisfy the general assumptions 
of the results posted in \cite{CiaMelSch19}.

We first formulate some abstract conditions on the operator $\op$, the data $v^0$ and $\vd$, and the functional $\F_{[d]}^{\alpha,\vd}$, defined in \autoref{eq:reg}. For the definition of the constant $M$ associated to $\op$ see \autoref{eq:M}.
\begin{assumption} \label{ass:2} Let \autoref{ass:gen} hold. 
	Moreover, let $v^0 \in L^p(\Sigma;\R^M)$ and let $\rho$ be a mollifier, which satisfies the separation property \autoref{eq:sep}. 
	We assume that 
	\begin{enumerate}
		\item $\op:W^{s,p}(\Omega;K) \to L^p(\Sigma;\R^M)$ is well--defined and sequentially continuous with respect to the weak topology on 
		$W^{s,p}(\Omega;\R^m)$, that is if $(w_n)_{n\in \N}$ is a sequence in $W^{s,p}(\Omega;K)$ converging weakly to 
		$w^* \in W^{s,p}(\Omega;K)$ (with respect to $W^{s,p}(\Omega;\R^m)$), then it holds that $\op[w_n] \rightarrow \op[w^*]$ 
		strongly in $L^p(\Sigma;\R^M)$.\label{itm:F1}
		\item	 
		For every $t > 0$ and $\alpha > 0$ the level sets
		\begin{equation*}
			\text{level}(\F^{\alpha,v^0}_{[d]};t) \defeq \set{  w \in W^{s,p}(\Omega;K):\  \F^{\alpha,v^0}_{[d]}(w) \leq t  }  
		\end{equation*}
		are weakly sequentially pre-compact in $W^{s,p}(\Omega;\R^m)$. \label{itm:F2}
		\item	  
		There exists $\bar{t} > 0$ such that $\text{level}_{\bar{t}}(\F^{\alpha,v^0}_{[d]})$ is nonempty. \label{itm:F3}
	\end{enumerate}
\end{assumption}
Moreover, we make the following definition.
\begin{definition}\label{def:MNS} Let \autoref{ass:2} hold.  
	Every element $w^* \in W^{s,p}(\Omega;K)$ satisfying
	\begin{equation*}
	    \boxed{
        \Reg_{[\d]}^l(w^*) = \inf \set{ \Reg_{[\d]}^l (w):\ w \in W^{s,p}(\Omega;K), \op[w] = v^0 },}
	\end{equation*}
	is called a 
	\emph{$\Reg_{[\d]}^l$-minimizing solution} of \autoref{eq:op}.
\end{definition}
Under basic assumptions the existence of a $\Reg_{[\d]}^l$-minimizing solution is guaranteed. 
The proof is analogous to the proof of \cite[Lemma 3.2]{SchGraGroHalLen09}. 
\begin{lemma}\label{lem:MNS}
	Let \autoref{ass:2} hold and assume that there exists a solution of \autoref{eq:op} in $W^{s,p}(\Omega;K)$. 
	Then there exists a $\Reg_{[\d]}^l$-minimizing solution.
\end{lemma}
According to \cite{CiaMelSch19} we now have the following result.
\begin{theorem}\label{th:ex} 
  Let \autoref{ass:2} hold. Moreover, we assume that there exists a solution of 
   \autoref{eq:op} (\autoref{lem:MNS} then guarantees the existence of a $\Reg_{[\d]}^l$-minimizing solution $w^\dagger \in W^{s,p}(\Omega;K)$).
   Then the following results hold:
	\begin{description} 	 
		\item{\emph{Existence:}} For every $v \in L^p(\Sigma;\R^M)$ and $\alpha > 0$ the functional 
		 $\F_{[d]}^{\alpha,v}: W^{s,p}(\Omega;K) \rightarrow [0, \infty)$ attains a minimizer in $W^{s,p}(\Omega;K)$. 
		\item{\emph{Stability:}} Let $\alpha > 0$ be fixed, $\vd \in L^p(\Sigma;\R^M)$ and let $(v_k)_{k \in \N}$ be a sequence 
		in $L^p(\Sigma;\R^M)$ such that $\norm{\vd-v_k}_{L^p(\Sigma;\R^M)} \rightarrow 0$. Then every sequence 
		$(w_k)_{k \in \N}$ satisfying 
		\begin{equation*}
			w_k \in \arg \min \set{ \F_{[d]}^{\alpha,v_k}(w):\ w \in W^{s,p}(\Omega;K) }
		\end{equation*}	
		has a converging subsequence with respect to the weak topology of $W^{s,p}(\Omega;\R^m)$. 
		The limit $\tilde{w}$ of every such converging subsequence $(w_{k_j})_{j \in \N}$ is a minimizer of 
		$\F_{[d]}^{\alpha,\vd}$. Moreover, $(\Reg_{[\d]}^l(w_{k_j}))_{j \in \N}$ converges to $\Reg_{[\d]}^l(\tilde{w})$.	
		
		\item{\emph{Convergence:}} Let $\alpha:(0, \infty) \rightarrow (0,\infty)$ be a function satisfying $\alpha(\delta) \rightarrow 0$ and 
		$\frac{\delta^p}{\alpha(\delta)} \rightarrow 0$ for $\delta \to 0$.
		
		Let $(\delta_k)_{k \in \N}$ be a sequence of positive real numbers converging to $0$. Moreover, let 
		$(v_k)_{k \in \N}$ be a sequence in $L^p(\Sigma;\R^M)$ with $\norm{v^0-v_k}_{L^p(\Sigma;\R^M)} \leq \delta_k$ and 
		set $\alpha_k \defeq \alpha(\delta_k)$. Then every sequence 
		$$\left(w_k \in \arg \min \set{ \F^{\alpha_k,v_k}_{[d]}(w):\ w \in W^{s,p}(\Omega;K)} \right)_{k \in \N}$$ 
		has a \emph{weakly converging} subsequence 
		$w_{k_j} \rightharpoonup \tilde{w}$ as $j \to \infty$ (with respect to the topology of $W^{s,p}(\Omega;\R^m)$), and the limit 
		$\tilde{w}$ is a $\Reg_{[\d]}^l$-minimizing solution. 
		In addition, $\Reg_{[\d]}^l(w_{k_j}) \rightarrow \Reg_{[\d]}^l(\tilde{w})$. Moreover, if $w^\dagger$ is unique, then it follows that 
		$w_k \rightharpoonup w^\dagger$ weakly (with respect to the topology of $W^{s,p}(\Omega;\R^m)$) and 
		$\Reg_{[\d]}^l(w_{k}) \rightarrow \Reg_{[\d]}^l(w^\dagger)$.
	\end{description}
\end{theorem}

For the Radon transform, $\op = \radon$ (see \autoref{def:Radon} below), the solution of \autoref{eq:op} is unique \cite{Nat01}, 
guaranteeing that we have convergence of the \emph{whole} sequence $(w_k)_{k \in \N}$ in \autoref{th:ex}.

Uniqueness of a $\Reg_{[\d]}^l$-minimizing solution of \autoref{eq:op} also holds true in case $\op = \ray$ (see \autoref{rem:RayTransform1} below) 
when $\d = \d_{\R^2}$ on $W^{s,p}(\Omega;\R^2)$:
To see this, let $w_1,w_2 \in W^{s,p}(\Omega;\R^2), w_1 \neq w_2$ be two different $\Reg_{[\d_{\R^2}]}^l$-minimizing solutions 
as in \autoref{def:MNS}. 
The linearity of $\ray$ ensures that $w_1$ and $w_2$ are not linearly dependent 
because otherwise $w_1 = \lambda w_2$ and thus 
\begin{equation*}
\ray[w_1]  = \ray[\lambda w_2] = \lambda \ray[w_2] = v^0,
\end{equation*}
which is only possible if $\lambda = 1$. 
Moreover, $w_1$ and $w_2$ are not equal up to a constant: Assume on the contrary, $w_1 = w_2 + c$. Then 
\begin{equation*}
\ray[w_1] = \ray[w_2] + \ray[c] = v^0,
\end{equation*}
which implies $\ray[c] = 0$ and in turn implies $c = 0$, see \autoref{rem:RayTransform1} below, i.e. 
the null-function is the only constant function lying in the kernel of $\ray$, \cite{NatWue01}.

The linearity of $\ray$ also gives that $\ray[\frac{1}{2}(w_1+w_2)] = v^0$, 
which shows that $(w_1+w_2)/2$ is also a solution (note that we assume that $K=\R^2$). 
Moreover, from \autoref{re:RP} it follows that
\begin{equation*} \begin{aligned}
\Reg_{[\d_{\R^2}]}^l \left( \frac{1}{2}w_1 + \frac{1}{2} w_2 \right)
& < \left( \frac{1}{2} \big(\Reg_{[\dRM]}^l(w_1) \big)^{\frac{1}{p}} + \frac{1}{2} \big(\Reg_{[\d_{\R^2}]}^l(w_1) \big)^{\frac{1}{p}} \right)^p \\
& \leq 2^{p-1} \left( \frac{1}{2^p}\Reg_{[\d_{\R^2}]}^l(w_1) + \frac{1}{2^p}\Reg_{[\d_{\R^2}]}^l(w_1) \right) \\
&= \frac{1}{2} \Reg_{[\d_{\R^2}]}^l(w_1) + \frac{1}{2}\Reg_{[\d_{\R^2}]}^l(w_2) \\
&= \min \set{  \Reg_{[\d_{\R^2}]}^l(w):\ w \in W^{s,p}(\Omega;\R^2),\ray[w] = v^0  },
\end{aligned} \end{equation*}
where the first inequality is strict because $w_1$ and $w_2$ are linearly independent and not equal up to a constant. 
This yields that $w_1$ and $w_2$ must be equal. \\
In the numerical examples in \autoref{sec:5} we consider particular subsets $K$ with associated metrics $d$, see \autoref{subsec:vfdata}, and \emph{not}  $K = \R^2$ and $\d = \d_{\R^2}$. For this case we do not know if the \emph{$\Reg_{[\d]}^l$}-minimizing solution is unique.

\section{Regularization of the Radon and Ray Transform Inversion} \label{sec:3}
In this section we verify \autoref{ass:2}, in the case $\op = \mathcal{R}, \mathcal{D}$, respectively, such that 
\autoref{lem:MNS} and \autoref{th:ex} are applicable. For computational purposes it is convenient to identify every function 
$w \in L^p(\Omega;\R^m)$ with its extension by $0$ outside of $\Omega$.

\begin{definition}[Radon transform] \label{def:Radon}
	For given $(r,\varphi) \in \Sigma$ let 
	$$H^{n-1}_{r,\varphi} \defeq \set{  x \in \R^n:\ x \cdot \varphi = r  }$$
	denote the $n-1$-dimensional hyperplane with orientation $\varphi$ and distance $r$ from the origin.
	The \emph{Radon transform}  
	$\radon:L^p(\Omega;\R^m) \rightarrow L^p(\Sigma;\R^m)$ computes the componentwise averages of a function $w$ over these hyperplanes and thus 
	is given by
	\begin{equation}\label{eq:RadonTransform}\boxed{
		\radon[w] (r,\varphi) \defeq \int_{H^{n-1}_{r,\varphi}} w(x) \dH(x),}
	\end{equation}
	where the transformation is understood component-wise, that is
	$\radon[w]  \defeq (\radon[w]_1, \dots, \radon[w]_m)^T$.
	Here, we denote by $\dH$ the $(n-1)$-dimensional Hausdorff measure on the hyperplane $H^{n-1}_{r,\varphi}$.
			
\end{definition}
Now we recall that $\radon$ is indeed a well-defined and sequentially continuous operator from $W^{s,p}(\Omega;K) \subseteq L^p(\Omega;\R^m)$ to $L^p(\Sigma;\R^m)$ guaranteeing that Item \ref{itm:F1} of \autoref{ass:2} holds true.  
\begin{theorem} \label{th:RadProp} Let \autoref{ass:gen} hold.
	\begin{itemize}
		\item[$\bullet$] Assume that $w \in L^p(\Omega;\R^m)$. 
		Then $\radon[w] \in L^p(\Sigma;\R^m)$ and there exists a constant $C := C(R,n,p)$ such that
		\begin{equation}\label{eq:RadonContinuity}
			\norm{\radon[w]}_{L^p(\Sigma;\R^m)}^p \leq C \norm{w}_{L^p(\Omega;\R^m)}^p.
		\end{equation}
		\item[$\bullet$] The Radon transform
		is sequentially continuous with respect to the weak topology on
		$W^{s,p}(\Omega;\R^m)$,
		meaning that for a sequence  
		$(w_k)_{k\in \N}$ in $W^{s,p}(\Omega;K) \subseteq L^p(\Omega;\R^m)$ converging weakly to 
		$w_* \in W^{s,p}(\Omega;K)$ (with respect to $W^{s,p}(\Omega;\R^m)$) it holds that $\radon[w_k] \rightarrow \radon[w_*]$ strongly in $L^p(\Sigma,\R^m)$. 		
	\end{itemize}
\end{theorem}

\begin{proof}
	The first item is obtained by direct calculation using H\"olders inequality. The arguments are similar to the ones in \cite{Qui06}, where the case $p=1$ and $n=2$ is investigated.
	
    We show the desired inequalities for $m=1$ first: Denoting by $\chi_\Omega$ the characteristic function 
    of $\Omega$ and by $\Gamma$ the Gamma function we get for every $w \in L^p(\Omega;\R)$
	\begin{equation*} \begin{aligned}
		\norm{\radon[w]}_{L^p(\Sigma;\R)}^p &
  		= \int_{\mathbb{S}^{n-1}} \int_{0}^{R} \enorm{ \int_{H^{n-1}_{r,\varphi}}  w(x) \dH(x) }^p \,\mathrm{d}r \dH(\varphi) \\
 		& = \int_{\mathbb{S}^{n-1}} \int_{0}^{R} \enorm{ \int_{H^{n-1}_{r,\varphi}} \chi_{\Omega}(x) w(x) \dH(x) }^p \,\mathrm{d}r \dH(\varphi) \\
        & \leq \int_{\mathbb{S}^{n-1}} \int_{0}^{R} ( \mathcal{H}^{n-1}(\mathcal{B}^{n-1}_R(0)))^{p-1} 
        \int_{H^{n-1}_{r,\varphi}} \enorm{w(x)}^p \dH(x) \,\mathrm{d}r \dH(\varphi) \\
 		& \leq (\mathcal{H}^{n-1}(\mathcal{B}^{n-1}_R(0)))^{p-1} \mathcal{H}^{n-1}(\mathbb{S}^{n-1}) \norm{w}_{L^p(\Omega;\R)}^p \\
 		&= \left( \frac{\pi^{\frac{n-1}{2}}}{\Gamma(\frac{n-1}{2} + 1)}R^{n-1} \right)^{p-1}
 		\frac{n\pi^{\frac{n}{2}}}{\Gamma(\frac{n}{2} + 1)}
 		\norm{w}_{L^p(\Omega;\R)}^p
 		\eqdef C(R,n,p) \norm{w}_{L^p(\Omega;\R)}^p,
	\end{aligned} \end{equation*}
    where $\mathcal{H}^k$ denotes the $k$-dimensional Hausdorff measure.
	The claim for $m > 1$ follows by the equivalence of norms in $\R^m$. 	
	
	To prove the second item let $(w_k)_{k\in \N}$ be a sequence in $W^{s,p}(\Omega;K) \subseteq L^p(\Omega;\R^m)$ with $w_k \rightharpoonup w_*$ 
	weakly as $k \to \infty$ with respect to the $W^{s,p}(\Omega;\R^m)$ topology.  
	Then, from \autoref{lem:inclusion} it follows that $w_* \in W^{s,p}(\Omega;K)$ and that
	$w_k \rightarrow w_*$ strongly in $L^p(\Omega;\R^m)$. The assertion then follows from \autoref{eq:RadonContinuity}. 
\end{proof}

In the following we recall the definition of the ray transform. See \cite{Sch08} for more 
information on this transform and applications.

\begin{definition}[2D Ray transform]\label{rem:RayTransform1}
	Let $n=m=2$ and define $\theta \defeq \theta(\varphi) = (\cos(\varphi), \sin(\varphi))^T \in \sphere$ and
	$\theta^{\perp} \defeq \theta^{\perp}(\varphi) = (-\sin(\varphi), \cos(\varphi))^T \in \sphere$.
	Then the \emph{2D ray transform} $\ray:L^p(\Omega;\R^2) \rightarrow L^p(\Sigma;\R)$ is defined as follows (cf.\cite{Sch08}):
	\begin{equation*} 
	\boxed{
		\ray[w] (r,\varphi) \defeq \int_{H^{1}_{r,\varphi}} w(x) \cdot \theta^{\perp} \dH(x) 
		= \int_{-\infty}^{\infty}  w(r\theta + t \theta^{\perp}) \cdot \theta^{\perp}\mathrm{d}t.}
	\end{equation*}
\end{definition}
The 2D ray transform is related to the Radon transform by the following identity
$$\ray[w] (r,\varphi) = \radon[w] (r,\varphi) \cdot \theta^{\perp},$$ 
and therefore properties of $\ray$ can be inherited from $\radon$:
\begin{theorem} \label{th:RayProp} Let \autoref{ass:gen} hold.
	\begin{itemize}
		\item[$\bullet$] There exists a constant $C:=C(\Omega,n,p)$ such that for all $w \in L^p(\Omega;\R^2)$,
		\begin{equation}\label{eq:RayContinuity}
			\norm{\ray[w]}_{L^p(\Sigma;\R)}^p \leq C \norm{w}_{L^p(\Omega;\R^2)}^p.
		\end{equation}
		\item[$\bullet$] The 2D ray transform is sequentially continuous with respect to the weak topology on
		$W^{s,p}(\Omega;\R^2)$. 
		That is, for a sequence  
		$(w_k)_{k\in \N}$ in $W^{s,p}(\Omega;K) \subseteq L^p(\Omega;\R^2)$ converging weakly to 
		$w_* \in W^{s,p}(\Omega;K)$ (with respect to $W^{s,p}(\Omega;\R^2)$) it holds that $\ray[w_k] \rightarrow \ray[w_*]$ strongly in $L^p(\Sigma,\R)$.
	\end{itemize}  
\end{theorem}

Our goal is to show that the functional defined in \autoref{eq:reg} with $\op = \radon$ and $\op = \ray$, respectively, fulfills \autoref{ass:2}.
The first item of \autoref{ass:2} has been proven already in \autoref{th:RadProp} and in \autoref{th:RayProp}, respectively, and now we state 
a result which gives the second assertion of \autoref{ass:2}. 

Therefore we first need the following lemma.

\begin{lemma}\label{lem:meanbound}

	Let \autoref{ass:gen} hold and let $\op = \radon$ or $\op = \ray$, with associated dimension M as in \autoref{eq:M}, respectively.
	In addition assume that $\vd \in L^p(\Sigma;\R^M)$. 
	Let $w \in \level{\F_{[d]}^{\alpha,\vd}}{t} := \set{w \in W^{s,p}(\Omega;K) : \F_{[d]}^{\alpha,\vd}(w) \leq t}$.
	
	Then, for every $t > 0$ there exists a constant $C_M > 0$ such that
	\begin{equation}\label{eq:meanbound}
	\norm{\op [\overline{w}]}_{L^p(\Sigma;\R^M)} \leq C_M,
	\end{equation}
	where $\bar{w}$ is the component-wise mean average of $w$ as defined in \autoref{eq:PT}.
\end{lemma}
\begin{proof}
	Let $w \in \level{\F_{[d]}^{\alpha,\vd}}{t}$, then it follows from the definition of $\F_{[d]}^{\alpha,\vd}$ that
	\begin{equation}\label{eq:boF}
	\int_\Sigma \enorm{\op[w](r,\varphi)- \vd(r,\varphi)}^p \drp \leq t \text{ and } 
	\Reg_{[\d]}^l(w) 
	\leq \frac{t}{\alpha}.
	\end{equation}
	The first inequality ensures that 
	\begin{equation} \label{eq:boundOnRadon}
	\norm{\op [w]}_{L^p(\Sigma;\R^M)} \leq \norm{\op [w]-\vd}_{L^p(\Sigma;\R^M)} + \norm{\vd}_{L^p(\Sigma;\R^M)} \leq t^{\frac{1}{p}} + \norm{\vd}_{L^p(\Sigma;\R^M)} \eqdef t_0 < \infty. 
	\end{equation}
	Therefore,
	\begin{equation} \label{eq:t0} \begin{aligned}
	t_0^2 &\geq \norm{\op [w]}_{L^p(\Sigma;\R^M)}^2
	= \norm{\op [w- \overline{w}] + \op [\overline{w}] }_{L^p(\Sigma;\R^M)}^2\\
	&\geq \left( \norm{\op [w- \overline{w}] }_{L^p(\Sigma;\R^M)} - \norm{\op [\overline{w}] }_{L^p(\Sigma;\R^M)}  \right)^2 \\
	& \geq \norm{\op [\overline{w}] }_{L^p(\Sigma;\R^M)} \left( 
	\norm{\op [\overline{w}] }_{L^p(\Sigma;\R^M)} - 2 \norm{\op [w- \overline{w}] }_{L^p(\Sigma;\R^M)} \right) \\
	&\geq \norm{\op [\overline{w}] }_{L^p(\Sigma;\R^M)} \left( 
	\norm{\op [\overline{w}] }_{L^p(\Sigma;\R^M)} - 2 \norm{\op} \norm{w- \overline{w} }_{L^p(\Omega;\R^M)}
	\right),
	\end{aligned} \end{equation}
	where we used the linearity of $\op$. By $\norm{\op}$ we denote the operator norm of $\op: W^{s,p}(\Omega;K) \to L^p(\Sigma;\R^M)$.
	
	Defining $r \defeq \norm{\op [\overline{w}] }_{L^p(\Sigma;\R^M)}$ and $z \defeq \norm{\op} \norm{w- \overline{w}}_{L^p(\Omega;\R^M)}$ 
	it follows from \autoref{eq:PT} and \autoref{eq:boF} that
    \begin{equation} \label{eq:t1} \begin{aligned}
	0 \leq z &= \norm{\op} \norm{w-\overline{w}}_{L^p(\Omega;\R^M)} 
	\leq \norm{\op} C_P^{1/p} (\Reg_{[\d]}^l(w))^{1/p} \\
	&\leq 
	\norm{\op} C_P^{1/p} \left( \frac{t}{\alpha} \right)^{1/p} \eqdef t_1 < \infty.
	\end{aligned}
	\end{equation}
	Then by using the equivalence of the relations
	\begin{equation*}
	0 \leq r \ , \ r(r-2z) \leq t_0^2 \quad \text{ and } 0 \leq r \leq \sqrt{t_0^2 + z^2} + z, 
	\end{equation*}	
	from \autoref{eq:t1} and \autoref{eq:boundOnRadon} we conclude that 
	\begin{equation}\label{eq:BoundOfMean}
	r = \norm{\op [\overline{w}] }_{L^p(\Sigma;\R^M)} 
	\leq \sqrt{t_0^2 + t_1^2} + t_1  \eqdef C_M < \infty.
	\end{equation}
\end{proof}

The following coercivity result is important in guaranteeing the second assertion of \autoref{ass:2}.

\begin{theorem}[Coercivity]\label{th:Coercivity} 
	Let \autoref{ass:gen} hold and let $\op = \radon$ or $\op = \ray$ with associated dimension M as in \autoref{eq:M}, respectively.
	In addition assume that $\vd \in L^p(\Sigma;\R^M)$.
	The functional $\F_{[d]}^{\alpha,\vd}$ defined in \autoref{eq:reg} with operators 
	$\op = \radon$ and $\op = \ray$, respectively, is $W^{s,p}$-coercive, that is
	for every $t > 0$ there exists $C > 0$ such that
	\begin{equation}\label{eq:coercivity}
     \norm{w}_{W^{s,p}(\Omega;\R^m)} \leq C \text{ for all } w \in \level{\F_{[d]}^{\alpha,\vd}}{t}.
	\end{equation}
\end{theorem}

\begin{proof}  
 We only prove the result for $\op = \radon$, for $\op = \ray$ the proof can be done in a similar way. \\
 The first inequality in \autoref{eq:boF} implies (as in the proof of \autoref{lem:meanbound}) \autoref{eq:boundOnRadon}.
	
 We start by bounding the $L^p$-norm of $w \in \level{\F_{[d]}^{\alpha,\vd}}{t}$: From Jensen's inequality it follows that
	\begin{equation} \label{eq:EWM}
		\norm{w}_{L^p(\Omega;\R^m)}^p \leq \left(\norm{w-\overline{w} }_{L^p(\Omega;\R^m)} + \norm{\overline{w} }_{L^p(\Omega;\R^m)} \right)^p
		\leq 2^{p-1} \left(\norm{w-\overline{w} }_{L^p(\Omega;\R^m)}^p + \norm{\overline{w} }_{L^p(\Omega;\R^m)}^p \right),
	\end{equation}
	where $\bar{w}$ is as defined in \autoref{eq:PT}.
	Due to the linearity of the Radon transform we know that 
    $\norm{\radon [\overline{w}]}_{L^p(\Sigma;\R^m)}^p = 
		\enorm{\overline{w}}^p \norm{\radon [\chi_{\Omega}]}_{L^p(\Sigma;\R)}^p$, or in other words
    \begin{equation} \label{eq:RMB}
	  \begin{aligned}
	    \frac{1}{\abs{\Omega}} \norm{\overline{w}}_{L^p(\Omega;\R^m)}^p =
		\enorm{\overline{w}}^p = \norm{\radon [\chi_{\Omega}]}_{L^p(\Sigma;\R)}^{-p} \norm{\radon [\overline{w}]}_{L^p(\Sigma;\R^m)}^p 
		=: C_\radon \norm{\radon [\overline{w}]}_{L^p(\Sigma;\R^m)}^p.
	  \end{aligned}
	\end{equation}
	Inserting this into \autoref{eq:EWM} and using \autoref{eq:PT}, \autoref{eq:t1} and \autoref{eq:BoundOfMean} it follows that
	\begin{equation} \label{eq:bRLp}
	\norm{w}_{L^p(\Omega;\R^m)}^p 
	\leq 2^{p-1} \left(\norm{w-\overline{w}}_{L^p(\Omega;\R^m)}^p + \enorm{\overline{w}}^p|\Omega| \right)
	\leq 2^{p-1} \left(C_P \frac{t}{\alpha} + C_\radon C_M |\Omega| \right)  \eqdef \tilde{C} < \infty. 
	\end{equation}
	In order to prove \autoref{eq:coercivity} it remains to bound the fractional Sobolev semi-norm.
    Because of \autoref{eq:MI} and \autoref{eq:boF} we have $\Reg_{[\dRM]}^l \leq \Reg_{[\d]}^l \leq \frac{t}{\alpha}$.
   
    \begin{itemize}
     \item If $l=0$, then because $\Reg_{[\dRM]}^0 = \abs{\cdot}_{W^{s,p}(\Omega;\R^m)}$ if follows from \autoref{eq:bRLp} that
	\begin{equation}\label{eq:l0}
	 \norm{w}_{W^{s,p}(\Omega;\R^m)}^p = \norm{w}_{L^p(\Omega;\R^m)}^p + \Reg_{[\dRM]}^0(w) \leq \tilde{C} + \frac{t}{\alpha} =:C^0.
	\end{equation}
	 \item 
	If $l=1$, then from \autoref{le:Equiv}, it follows that
	\begin{equation}\label{eq:l1}
	\abs{w}_{W^{s,p}(\Omega;\R^m)}^p \leq	\norm{w}^p_{W^{s,p}(\Omega;\R^m)} \leq \overline{C} \left( \norm{w}^p_{L^p(\Omega;\R^m)} + \Reg_{[\d]}^1(w)  \right) \leq \overline{C} \left( \tilde{C} + \frac{t}{\alpha}  \right) =: C^1.
	\end{equation}
	\end{itemize}
To conclude we define $C \defeq \left( \tilde{C} +  \max\{ C^0,C^1 \} \right)^{1/p}$.
\end{proof}

To get the second assertion of \autoref{ass:2} we still need to connect the level-sets $\level{\F_{[d]}^{\alpha,\vd}}{t}$ and $\level{\F_{[d]}^{\alpha,v^0}}{t}$, see \cite[Lemma 3.5]{CiaMelSch19}. 
\begin{lemma}
	Let \autoref{ass:gen} be satisfied.
	For every $w \in W^{s,p}(\Omega,K)$ and $v_1, v_2 \in L^p(\Sigma, \R^M)$, where $M$ is the associated dimension to $\op = \radon$ or $\op = \ray$ as described in \autoref{eq:M}, it holds that
	\begin{equation}\label{eq:LS}
	\F_{[d]}^{\alpha,v_1}(w) \leq 2^{p-1} \left(
	\F_{[d]}^{\alpha,v_2}(w) + \|v_2-v_1\|^p_{L^p(\Sigma;\R^M)}
	\right).
	\end{equation}
\end{lemma}

The previous lemma and \autoref{th:Coercivity} shows that Item \ref{itm:F2} of \autoref{ass:2} is valid.
The third item is satisfied as well, what we see using $w \equiv const$ to get that $\F^{\alpha,v^0}_{[\d]} \not \equiv \infty$. 
The first item was shown in\autoref{th:RadProp} and in \autoref{th:RayProp}.
Thus, all items of \autoref{ass:2} are fulfilled and according to \cite{CiaMelSch19} we get existence of a minimizer, 
as well as a stability and convergence result.

\begin{corollary} \label{le:ExStCo}
	Let \autoref{ass:gen} be satisfied. The functional $\F_{[d]}^{\alpha,\vd}$ defined in \autoref{eq:reg} with $\op = \radon$ and $\op = \ray$, satisfies 
	the assertions of \autoref{th:ex}. That is, 
	$\F_{[d]}^{\alpha,\vd}$ attains a minimizer and fulfills a stability as well as a convergence result.
\end{corollary}

In the following we discuss different choices of $K$ and associated metrics $\d$.


\subsection{Examples of particular sets $K$ and metrics $\d$} 
\label{ss:k} 
For particular choices of the set $K$ and associated metrics $\d$ we show that the corresponding functional $\F_{[d]}^{\alpha,\vd}$ 
(as defined in \autoref{eq:reg}) is well-defined, attains a minimizer and is stable and convergent in the sense of \autoref{le:ExStCo}.
For the particular choices of $K$ and $\d$ we present numerical results in the subsequent \autoref{sec:5}.

\subsubsection{$\sphere$-valued data} 
\label{subsec: S1data} 
In this subsection we consider $$K \defeq \sphere \subseteq \R^2$$
with associated metric
\begin{equation} \label{eq:dS}
\boxed{
\dS(x,y) \defeq \arccos(x^T y) \text{ for all } x,y \in \sphere.}
\end{equation}
We start by making a few technical assumptions.
\begin{assumption} \label{as:sp}
	Let $\emptyset \neq \Omega \subset \R^n$ be a bounded and simply connected domain and let either 
	\begin{enumerate}
		\item $s \in (0, \infty)$ and $p \in (1,\infty)$ if $n = 1$ or
		\item $s \in (0,1)$, $p \in (1,\infty)$ satisfy  
		$sp < 1 \text{ or } sp \geq n$ if $n \geq 2$ \label{itm:sp}.  		
	\end{enumerate}
\end{assumption}
Under this assumption we have the following identity:
\begin{lemma}\cite[Theorem 1 \& 2]{BouBreMir00b} \label{lem:BBM}
	Let \autoref{as:sp} be satisfied. Then	
	every function $w \in W^{s,p}(\Omega;\sphere)$ can be represented as $w = \e^{\i u}$ 
	with $u \in W^{s,p}(\Omega;\R)$. The function $u$ is called a \emph{lifting} of $w$.
\end{lemma}
Moreover, we make the following definition.

\begin{definition}\label{def:lF}
	Let $\dS$ be defined as in \autoref{eq:dS}. We call
	\begin{equation}\label{eq:liftedFunc} 
	\tilde{\F}_{[\dS]}^{\alpha,\vd}(u) \defeq 
	\int_\Sigma \enorm{\op[\e^{\i u}](r,\varphi)- \vd(r,\varphi)}^2 \drp
	+ \alpha \Reg_{[\dS]}^1(\e^{\i u}) \text{ for } u \in W^{s,p}(\Omega;\R).
	\end{equation}
	the \emph{lifted functional} of $\F_{[\dS]}^{\alpha,\vd}$ (which is defined on $W^{s,p}(\Omega;\sphere)$). 
\end{definition}

The following lemma summarizes elementary facts and connects $\tilde{\F}_{[\dS]}^{\alpha,\vd}$ and $\F_{[\dS]}^{\alpha,\vd}$, In particular it reveals 
that both functionals are well-defined, and attain a minimizer:
\begin{lemma}\label{lem:FFT}
	\begin{enumerate}
	    \item \label{lem:WspLifting} Let $u \in W^{s,p}(\Omega;\R)$. Then $\e^{\i u} \in W^{s,p}(\Omega;\sphere)$.
		\item \label{itm: dSequiv} $\dS$ is an equivalent metric to $\d_{\R^2}\big|_{\sphere \times \sphere}$.
		\item Minimization of the functional $\F_{[\dS]}^{\alpha,\vd}$ over $W^{s,p}(\Omega;\sphere)$ is well--posed, stable and convergent in the sense 
		of \autoref{th:ex}.
		\item $\tilde{\F}_{[\dS]}^{\alpha,\vd}(u) \defeq \F_{[\dS]}^{\alpha,\vd}(\e^{\i u})$ for all $u \in W^{s,p}(\Omega;\R)$. 
		\label{itm:FT}
		\item Let \autoref{as:sp} be satisfied. Then $\tilde{\F}_{[\dS]}^{\alpha,\vd}$ is well-defined and attains a minimizer on $W^{s,p}(\Omega;\R)$. 
	\end{enumerate}
\end{lemma}

\begin{proof}
	\begin{enumerate}
	  \item The proof of the first item is a direct consequence of the inequality $|\e^{\i a} - \e^{\i b}| \leq |a-b|, \ a,b \in \R$.
		\item From elementary calculations we obtain that 
		$|x-y| \leq \dS(x,y) \leq \frac{\pi}{2}|x-y|$ for all $x,y \in \sphere \subseteq \R^2.$
		\item The third item follows from the first item and \autoref{le:ExStCo}.
		\item This is true by \autoref{def:lF} and the definition of $\F_{[\dS]}^{\alpha,\vd}$ in \autoref{eq:reg}.  
		\item For the proof of the last item we refer to \cite[lemma 6.6]{CiaMelSch19}.
	\end{enumerate}
\end{proof}

In the following we show that minimization of the lifted functional $\tilde{\F}_{[\dS]}^{\alpha,\vd}$ is an equivalent 
method to minimizing $\F_{[\dS]}^{\alpha,\vd}$. The advantage of minimizing over $W^{s,p}(\Omega;\R)$ is that it forms a 
\emph{vector space} (contrary to the \emph{set} $W^{s,p}(\Omega;\sphere)$). 
We make use of this in the numerical examples in \autoref{subsec:Ex1}.

\begin{lemma} \label{le:equiv}
	Let \autoref{as:sp} be satisfied. 
	\begin{enumerate}
		\item If $u^* \in \mathrm{argmin}_{u \in W^{s,p}(\Omega;\R)} \tilde{\F}_{[\dS]}^{\alpha,\vd}(u)$ then $w^* \defeq \e^{\i u^*} \in W^{s,p}(\Omega;\sphere)$ is a minimizer of $\F_{[\dS]}^{\alpha,\vd}$, that is  $w^* \in \mathrm{argmin}_{w \in W^{s,p}(\Omega;\sphere)} \F_{[\dS]}^{\alpha,\vd}(w)$.
		\item Let $w^* \in \mathrm{argmin}_{w \in W^{s,p}(\Omega;\sphere)} \F_{[\dS]}^{\alpha,\vd}(w)$. Then there exist a lifting $u^* \in W^{s,p}(\Omega;\R), \ w^* = \e^{\i u^*},$ such that $u^*$ is a minimizer of $\tilde{\F}_{[\dS]}^{\alpha,\vd}$, that is 
		$u^* \in \mathrm{argmin}_{u \in W^{s,p}(\Omega;\R)} \tilde{\F}_{[\dS]}^{\alpha,\vd}(u)$.	
	\end{enumerate}
\end{lemma}	

\begin{proof}
	\begin{enumerate}
		\item We first remark that $w^* \in W^{s,p}(\Omega;\sphere)$ due to \autoref{lem:FFT} \autoref{lem:WspLifting}. \\
		Now, if $u^*$ is a minimizer of $\tilde{\F}_{[\dS]}^{\alpha,\vd}$ we have that
		\begin{equation*}
		\tilde{\F}_{[\dS]}^{\alpha,\vd}(u^*) \leq \tilde{\F}_{[\dS]}^{\alpha,\vd}(v) \quad \forall v \in W^{s,p}(\Omega;\R).
		\end{equation*} 
		In particular for $w^* = \e^{\i u^*}$ it holds that
		\begin{equation}\label{eq:MinOfLifting}
		\F_{[\dS]}^{\alpha,\vd}(w^*) \leq  \F_{[\dS]}^{\alpha,\vd}(\e^{\i v}) \quad \forall v \in W^{s,p}(\Omega;\R)
		\end{equation}
		by definition of $\tilde{\F}_{[\dS]}^{\alpha,\vd}$, see \autoref{lem:FFT} \autoref{itm:FT}. \\
		Now, let $w \in W^{s,p}(\Omega;\sphere)$ be arbitrary. Due to \autoref{as:sp} we have that \autoref{lem:BBM} holds true and that there exists a lifting $v_w \in W^{s,p}(\Omega;\R)$ such that $w = \e^{\i  v_w}$. 
		\autoref{eq:MinOfLifting} is in particular valid for $v = v_w$. Using this and again \autoref{lem:FFT} \autoref{itm:FT} we get that 
		$$\F_{[\dS]}^{\alpha,\vd}(w^*) \leq \F_{[\dS]}^{\alpha,\vd}(\e^{\i v_w}) = \F_{[\dS]}^{\alpha,\vd}(w),$$
		that is $w^* \in \mathrm{argmin}_{w \in W^{s,p}(\Omega;\sphere)} \F_{[\dS]}^{\alpha,\vd}(w)$.
		\item The proof of the second item is done analogously.
	\end{enumerate}
\end{proof}

\subsubsection{Vector field data}
\label{subsec:vfdata}
In this subsection we consider for $r > \varepsilon > 0$ fixed the set 
$$K \defeq \mathcal{B}_r(0) \setminus \mathcal{B}^\circ_\varepsilon(0) \subseteq \R^2,$$ 
where $\mathcal{B}_r(0)$ denotes the \emph{closed} ball of radius $r$ and center $0$ and $\mathcal{B}^\circ_\varepsilon(0)$ denotes 
the open ball of radius $\varepsilon$ and center $0$. We have to exclude an $\varepsilon$-ball around $0$ because otherwise
defining a suitable distance $\d$ on $K$ destroys the equivalence of $\d$ and $\d_{\R^2}\big|_{K \times K}$. 

An element $x \in K$ can be represented by its $1$-normalized orientation $\Theta \in [0,1)$ 
(that is its angle divided by $2\pi$) and by its length $l = \enorm{x} \in [\varepsilon,r]$. That is, 
\begin{center}
	to every $x \in K$ there exists 
	$(\Theta,l) \in  [0,1) \times [\varepsilon, r]$ such that $x= l \e^{2 \pi \i \Theta}$. 
\end{center}

\begin{lemma} \label{lem:ALM}
	Let $\gamma \geq 0, p > 1$ and define
	\begin{equation}\label{eq:vfMetrik}
		\boxed{\d(x_1,x_2) = \big( \dS^{p}(\e^{\i 2 \pi \Theta_1},\e^{\i 2 \pi \Theta_2}) + \gamma \enorm{ l_1- l_2}^p\big)^{1/p}, 
		\qquad x_1,x_2 \in K.}
	\end{equation}
	Then $d$ defines a metric on $K$ which is equivalent to $\d_{\R^2}\big|_{K \times K}$, i.e. there exist constants $C_u > c_l > 0$ such that
	\begin{equation*}
		c_l \enorm{x_1-x_2} \leq \d(x_1,x_2) \leq C_u \enorm{x_1-x_2}, \ x_1,x_2 \in K.
	\end{equation*}
	In particular $d/c_l$ is a normalized equivalent metric in the sense of \autoref{ass:gen}.
\end{lemma}

\begin{proof}
	We start by showing that $\d$ indeed fulfills the axioms of a metric. Let therefore $x_1,x_2,x_3 \in K$. 
	It is obvious that $\d(x_1,x_1) = 0$ and $\d(x_1,x_2) = \d(x_2,x_1)$. 
	Moreover $\d(x_1,x_2) = 0$ implies $x_1 =x_2$: If $\d(x_1,x_2) = 0$ both summands need to be zero, i.e. $\e^{\i  2 \pi \Theta_1} = \e^{\i  2 \pi \Theta_2}$ and $l_1 = l_2$, which gives $x_1 = x_2$. 
	To prove the triangle inequality we calculate
	\begin{align*}
		\d(x_1,x_3) &= \big( \dS^{p}(\e^{\i  2 \pi \Theta_1},\e^{\i  2 \pi \Theta_3}) + \gamma | l_1- l_3 |^p\big)^{1/p} \\
		&\leq \bigg( \big( \dS(\e^{\i  2 \pi \Theta_1},\e^{\i  2 \pi \Theta_2}) + \dS(\e^{\i  2 \pi \Theta_2},\e^{\i  2 \pi \Theta_3}) \big)^{p} + \big( \gamma | l_1- l_2 | + \gamma | l_2- l_3 | \big)^p  \bigg)^{1/p} \\
		&\leq \big( \dS^{p}(\e^{\i  2 \pi \Theta_1},\e^{\i  2 \pi \Theta_2}) + \gamma | l_1- l_2 |^p\big)^{1/p} 
		+ \big( \dS^{p}(\e^{\i  2 \pi \Theta_2},\e^{\i  2 \pi \Theta_3}) + \gamma | l_2- l_3 |^p\big)^{1/p} \\
		&= \d(x_1,x_2) + \d(x_2,x_3).
	\end{align*}
	
	Now we show the equivalence of $\d$ and $\d_{\R^2}\big|_{K \times K}$.
	We start with the lower bound, where we use \autoref{lem:FFT} \autoref{itm: dSequiv} and Jensen's inequality
	\begin{align*}
		\enorm{x_1-x_2} &= \enorm{l_1 \e^{i 2 \pi \Theta_1} - l_2 \e^{i 2 \pi \Theta_2}}
		\leq \enorm{ (l_1-l_2) \e^{i 2 \pi \Theta_1} + l_2 (\e^{i 2 \pi \Theta_1} - \e^{i 2 \pi \Theta_2}) } \\
		&\leq \enorm{l_1 - l_2} + l_2 \dS(\e^{\i  2 \pi \Theta_1},\e^{\i  2 \pi \Theta_2}) \\
		& \leq \max\{ 1,l_2 \} \left( 2^{p-1} ( \dS^p(\e^{\i  2 \pi \Theta_1},\e^{\i  2 \pi \Theta_2}) + \enorm{l_1 - l_2}^p ) \right)^{1/p} \\
		&= 2^{\frac{p-1}{p}} \max\{ 1,l_2 \} \d(x_1,x_2),
	\end{align*}
	giving $c_l \defeq \left( 2^{\frac{p-1}{p}} \max\{ 1,l_2 \}  \right)^{-1} > 0$. \\
	For the upper bound we remark that because we exclude $\mathcal{B}^\circ_\varepsilon(0)$ from the set $K$ there exists a constant $c_\varepsilon > 0$ (depending on $\varepsilon$) such that
	$ \dS(\e^{i 2 \pi \Theta_1} , \e^{i 2 \pi \Theta_2}) \leq \frac{\pi}{2} \enorm{\e^{i 2 \pi \Theta_1} -  \e^{i 2 \pi \Theta_2}} \leq \frac{\pi}{2} c_\varepsilon \enorm{l_1 \e^{i 2 \pi \Theta_1} - l_2 \e^{i 2 \pi \Theta_2}} $. 
	Then we compute using the sub-additivity of the p-th root 
	\begin{align*}
		\d(x_1,x_2) &= \big( \dS^{p}(\e^{\i  2 \pi \Theta_1},\e^{\i  2 \pi \Theta_2}) + \gamma | l_1- l_2 |^p\big)^{1/p} \\
		& \leq \dS(\e^{\i  2 \pi \Theta_1},\e^{\i  2 \pi \Theta_2}) + \gamma^{1/p} \enorm{l_1-l_2}  \\
		&\leq \frac{\pi}{2}  \enorm{\e^{\i  2 \pi \Theta_1}- \e^{\i  2 \pi \Theta_2}} 
		+  \gamma^{1/p} \enorm{ \enorm{l_1 \e^{\i  2 \pi \Theta_1}}-  
			\enorm{l_2 \e^{\i  2 \pi \Theta_2}} } \\
		& \leq \frac{\pi}{2} c_\varepsilon \enorm{l_1 \e^{i 2 \pi \Theta_1} - l_2 \e^{i 2 \pi \Theta_2}} + 	\gamma^{1/p}  \enorm{l_1 \e^{i 2 \pi \Theta_1} - l_2 \e^{i 2 \pi \Theta_2}} \\
		& \leq \max\{\frac{\pi}{2} c_\varepsilon, \gamma^{1/p}  \} \enorm{l_1 \e^{i 2 \pi \Theta_1} - l_2 \e^{i 2 \pi \Theta_2}} \\
		& \defeq C_u \enorm{x_1-x_2}.
	\end{align*}
\end{proof}

The previous \autoref{lem:ALM} yields that $K = \mathcal{B}_r(0) \setminus \mathcal{B}^\circ_\varepsilon(0) $ associated with the distance $\d$ in \autoref{eq:vfMetrik} fulfills \autoref{ass:gen} which implies \autoref{le:ExStCo}. Hence the corresponding functional $\F_{[\d]}^{\alpha,\vd}$ defined on $W^{s,p}(\Omega;K)$ is well-defined, attains a minimizer and fulfills a stability and convergence result.

In the next section we present numerical results using the subsets $K$ with associated metrics $\d$ defined above.
\section{Numerical Results}
\label{sec:5}

In this section we present numerical experiments for the reconstruction of (normalized) vector fields 
from Radon and ray transform data by using the regularization method consisting 
in minimizing the functional $\F_{[d]}^{\alpha,\vd}$ introduced in \autoref{eq:reg} with $l=1$
(taking advantage of the numerical efficiency effect of a mollifier) over $W^{s,p}(\Omega;K)$
for the particular choices of the subset $K$ and associated metric $\d$ as discussed in 
\autoref{ss:k}.

\subsection*{Numerical minimization}

For minimizing the functional $\F_{[d]}^{\alpha,\vd}$ with $l=1$ (introduced in \autoref{eq:reg}) 
we use a gradient descent algorithm with a backtracking line search algorithm, which ignores that 
values of the functions to be optimized lie in some set $K$. 
The steepest descent direction is given by the Gâteaux derivative of the functional $\F_{[d]}^{\alpha,\vd}$
defined  on $W^{s,p}(\Omega;\R^m)$.
The implementation is done in $\mathtt{Matlab}$.
The particular choice of the mollifier $\rho$ (see \autoref{ass:gen}) used to define the regularization 
functional $\Reg_{[\d]}^1$ has compact support. In the concrete implementation we use a (discrete) mollifier, 
which has non-zero entries on only either one, two or three neighboring pixels \emph{in each direction} around 
the center. We denote the number of non-zero elements by $n_{\rho}$. For an illustration see \autoref{fig:nrho}.
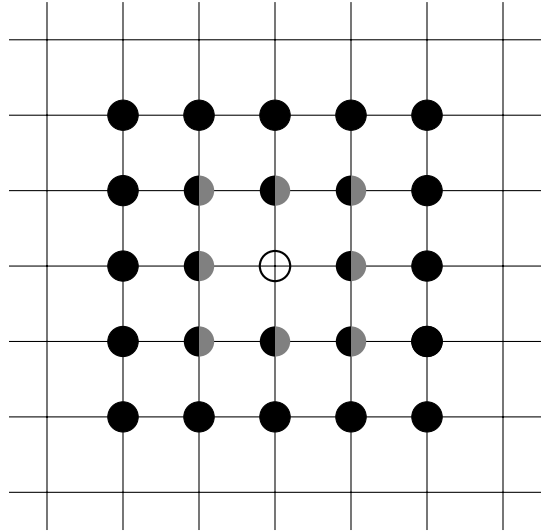
\begin{figure}[h]
	\begin{center}
		\begin{tikzpicture}
		\draw(-3.5,-3.5) grid (3.5,3.5);
		\draw[thick] (0,0) circle (0.2);
		\fill[black] (1,1) + (0, 0.2) arc (90:270:0.2);
		\fill[gray] (1,1) + (0, -0.2) arc (270:450:0.2);
		
		\fill[black] (1,0) + (0, 0.2) arc (90:270:0.2);
		\fill[gray] (1,0) + (0, -0.2) arc (270:450:0.2);
		
		\fill[black] (1,-1) + (0, 0.2) arc (90:270:0.2);
		\fill[gray] (1,-1) + (0, -0.2) arc (270:450:0.2);
		
		\fill[black] (0,-1) + (0, 0.2) arc (90:270:0.2);
		\fill[gray] (0,-1) + (0, -0.2) arc (270:450:0.2);
		
		\fill[black] (-1,-1) + (0, 0.2) arc (90:270:0.2);
		\fill[gray] (-1,-1) + (0, -0.2) arc (270:450:0.2);
		
		\fill[black] (-1,0) + (0, 0.2) arc (90:270:0.2);
		\fill[gray] (-1,0) + (0, -0.2) arc (270:450:0.2);
		
		\fill[black] (-1,1) + (0, 0.2) arc (90:270:0.2);
		\fill[gray] (-1,1) + (0, -0.2) arc (270:450:0.2);
		
		\fill[black] (0,1) + (0, 0.2) arc (90:270:0.2);
		\fill[gray] (0,1) + (0, -0.2) arc (270:450:0.2);
		
		\filldraw[black](0,2)circle[radius=0.2];
		\filldraw[black](1,2)circle[radius=0.2];
		\filldraw[black](2,2)circle[radius=0.2];
		\filldraw[black](2,1)circle[radius=0.2];
		\filldraw[black](2,0)circle[radius=0.2];
		\filldraw[black](2,-1)circle[radius=0.2];
		\filldraw[black](2,-2)circle[radius=0.2];
		\filldraw[black](1,-2)circle[radius=0.2];
		\filldraw[black](0,-2)circle[radius=0.2];
		\filldraw[black](-1,-2)circle[radius=0.2];
		\filldraw[black](2,-1)circle[radius=0.2];
		\filldraw[black](-2,-2)circle[radius=0.2];
		\filldraw[black](-2,-1)circle[radius=0.2];
		\filldraw[black](-2,0)circle[radius=0.2];
		\filldraw[black](-2,1)circle[radius=0.2];
		\filldraw[black](-2,2)circle[radius=0.2];
		\filldraw[black](-1,2)circle[radius=0.2];
		\end{tikzpicture}
		\caption{Support of the discrete mollifier $\rho$ with $n_{\rho} = 1$ (gray) and $n_{\rho} = 2$ (black) around the center point.}
		\label{fig:nrho}
	\end{center}
\end{figure}

In the following we use the subsequent projection operators for backprojection of functions with values in $\R^m$ onto $K$. 

\begin{definition}
 \begin{itemize}
  \item Let $K = \mathcal{B}_r(0)\setminus \mathcal{B}^\circ_\varepsilon(0) \subseteq \R^2$, then we consider the associated \emph{projection operator} 
  \begin{equation}\label{eq:ProjectionOperator}
   \begin{aligned}
    P: \R^2 &\to K, \qquad 
    x &\mapsto 
    \begin{cases}
        \frac{\varepsilon}{\enorm{x}} x & \text{if } 0 < \enorm{x} < \varepsilon \\
    	x & \text{if } \varepsilon \leq \enorm{x} \leq r \\
    	\frac{r}{\enorm{x}} x & \text{if } r < \enorm{x}  
    \end{cases}.
   \end{aligned}
  \end{equation}
  \item If $K=\sphere$ is represented by its (scaled) angle $\theta \in [0,1)$
  we define   
  \begin{equation}\label{eq:ProjectionOperatorII}
   \begin{aligned}
    \bar{P}: \R &\to [0,1), \qquad  \theta \mapsto \theta \text{ modulo } 1.
   \end{aligned}
\end{equation}
 \end{itemize}
 \end{definition}

  Let us first describe the synthetic data generation, the quality measures for comparing different methods, and alternative comparison methods:
  \subsection*{Synthetic data generation} \label{ss:syn}
  We simulate noisy data $\vd$ by calculating the Radon and ray transform, respectively, of an ideal input data 
  $w^\dagger$ using $\mathtt{Matlabs}$ built-in function $\mathtt{radon}$ for each component and then add Gaussian noise 
  of a certain variance to each component of its sinogram.

 \subsection*{Comparison method}
  We compare the minimizers of the functional $\F_{[d]}^{\alpha,\vd}$ with the ones obtained by (vectorial) Sobolev semi-norm 
  regularization of order $p$. That is we compare the results with minimizers of the functional 
  \begin{equation} \label{eq:VSN}
   \F_{\text{VSN}}(w) \defeq \int_\Sigma \enorm{\op[w](r,\varphi)- \vd(r,\varphi)}^p \drp + \beta \int \limits_{\Omega} \enorm{\nabla w(x)}_F^p \dx
  \end{equation}
  over $w \in W^{s,p}(\Omega;\R^m)$ (see \cite[Theorem 2.4]{DiNPalVal12}).
  Note that here we do not take into account that the values of the functions $w$ lie in some set $K$. 
 
 \subsection*{Quality measure}
 As a measure of quality of reconstruction we computed the \textit{signal-to-noise-ratio (snr)}, which is defined as
 \begin{equation*}
	snr = 20 \log \bigg( \frac{\enorm{w^\dagger}_2}{\enorm{w^\dagger-w_{rec}}_2} \bigg),
 \end{equation*}
 where $w^\dagger$ and $w_{rec}$ denote the (discrete) ground truth and reconstructed data.
\bigskip\par\noindent
Now, we come to the particular numerical examples:
\subsection{Reconstruction of normalized vector fields from Radon data}
\label{subsec:Ex1}
The first example concerns the reconstruction of a normalized vector field function $W^{s,p}(\Omega;\sphere) \subseteq W^{s,p}(\Omega;\R^2), \Omega \subset \R^2,$ from Radon data: More precisely, we consider the equation 
\begin{equation*}
 \op[w] = \radon[w] = v, 
\end{equation*}
where $w \in W^{s,p}(\Omega;\sphere)$ and $v \in L^p(\Sigma;\R^2)$.

We assume that noisy data $v^\delta \in L^p(\Sigma;\R^2)$ of the true sinogram data $v \in L^p(\Sigma;\R^2)$ are available.
In the numerical test $v^\delta$ are synthetic data, which was generated by adding Gaussian noise with variance $\sigma^2 = 3$
to both components of the discrete approximations of $v$. These data are shown in \autoref{sfig:sfig2-d} and \autoref{sfig:sfig2-e}. 

As described in \autoref{subsec: S1data}, $w \in W^{s,p}(\Omega;\sphere)$ can be represented by a function $u \in W^{s,p}(\Omega;\R)$ if 
$sp < 1$ or $sp \geq 2$. Instead of minimizing $\F_{[\dS]}^{\alpha,\vd}$ we are minimizing the lifted functional $\tilde{\F}_{[\dS]}^{\alpha,\vd}$
as introduced in \autoref{eq:liftedFunc}. This change of formulation is justified by \autoref{le:equiv}.

The discrete ground truth image (representing a function in $W^{s,p}(\Omega;\sphere)$ via its angle) of size $100 \times 100$ is shown in \autoref{sfig:sfig2-a}. 
For visualization we used the gray-scaled-$\mathtt{jet}$ colormap provided in $\mathtt{Matlab}$, where the gray values of function values 
$0$ and $1$ are identified assuming that the angles are scaled to $[0,1)$.

The reconstructed image using the lifted metric double integral regularization, \autoref{eq:liftedFunc}, can be seen in \autoref{sfig:sfig2-b}. 
For the Sobolev semi-norm reconstruction, \autoref{eq:VSN}, see  \autoref{sfig:sfig2-c}.
We observe significant noise reduction in both cases. 
However, our regularization method just regularizes the angle respecting the periodicity of the data correctly.

\begin{figure}[!h]
	\centering
	\begin{subfigure}[h]{0.32\linewidth}
		\includegraphics[width=1\linewidth]{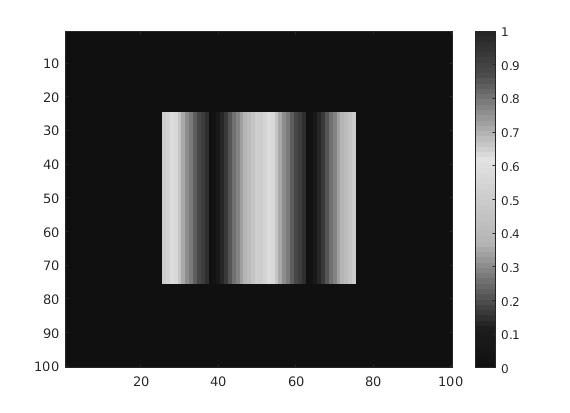}
		\caption{Original data $u$.\newline}
		\label{sfig:sfig2-a}
	\end{subfigure}
	\begin{subfigure}[h]{0.32\linewidth}
		\includegraphics[width=1\linewidth]{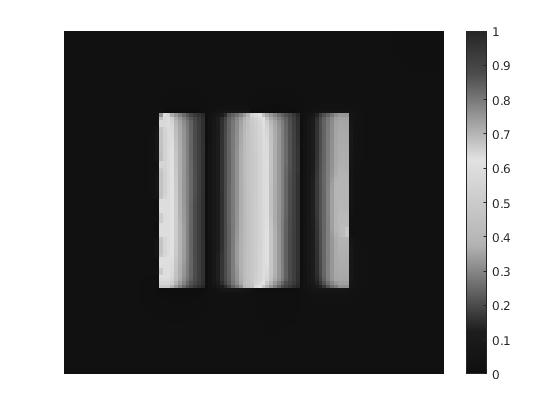}
		\caption{Result with metric double integral regularization, $snr = 23.64$.}
		\label{sfig:sfig2-b}
	\end{subfigure}
	\begin{subfigure}[h]{0.32\linewidth}
		\includegraphics[width=1\linewidth]{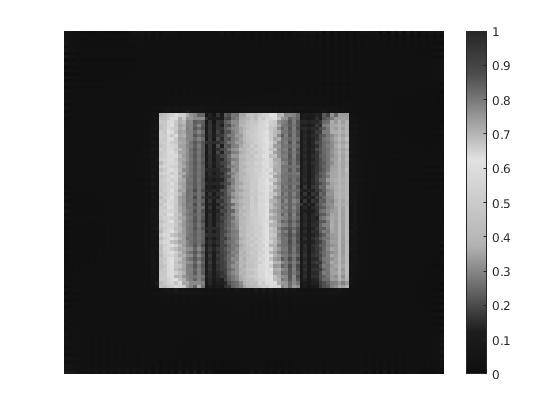}
		\caption{Sobolev semi-norm reconstruction, $snr = 18.70$.}
		\label{sfig:sfig2-c}
	\end{subfigure}
	
	\begin{subfigure}[h]{0.32\linewidth}
		\includegraphics[width=1\linewidth]{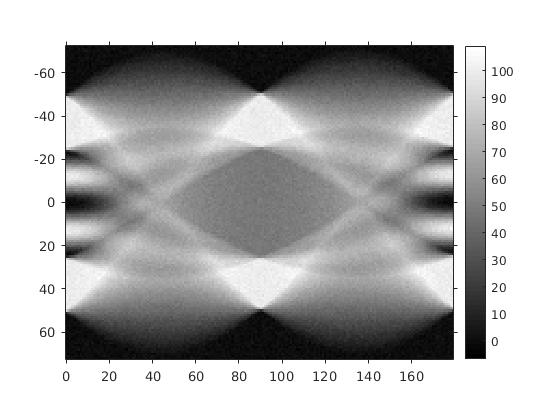}
		\caption{Noisy sinogram; first component of $v^\delta$. \newline}
		\label{sfig:sfig2-d}
	\end{subfigure}
	\begin{subfigure}[h]{0.32\linewidth}
		\includegraphics[width=1\linewidth]{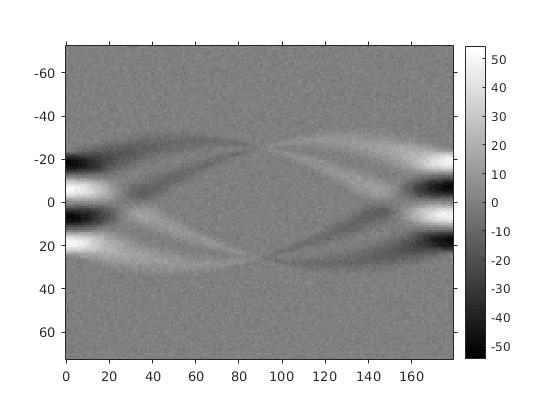}
		\caption{Noisy sinogram; second component of $v^\delta$. \newline}
		\label{sfig:sfig2-e}
	\end{subfigure}	
	\begin{subfigure}[h]{0.32\linewidth}
		\includegraphics[width=1\linewidth]{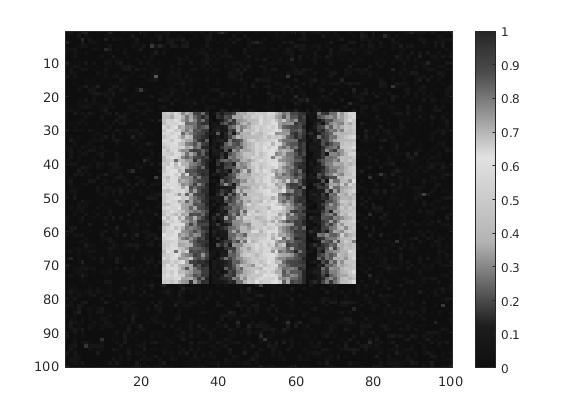}
		\caption{Starting image. \newline \quad \newline}
		\label{sfig:sfig2-f}
	\end{subfigure}
	\caption{First row: Original image. Result with metric double integral regularization with $p=1.1, s= 0.9, \alpha = 0.8, n_{\rho} = 1$. Sobolev semi-norm reconstruction with $p=1.1, \beta = 0.8$. Second row: Both components of the noisy sinogram data $v^\delta$. As starting function of our 
    iterative steepest descent algorithm we used the ground truth data perturbed by additive Gaussian noise with variance $\sigma^2 = 0.003$, see \autoref{sfig:sfig2-f}.}
	\label{fig:fig2}
\end{figure}
\clearpage

\subsection{Vector field reconstruction from ray data}
The second example consists of three numerical tests for reconstructing $2D$ vector field functions in $W^{s,p}(\Omega;K)$ with 
$K = \mathcal{B}_r(0)\setminus \mathcal{B}^\circ_\varepsilon(0) \subseteq \R^2, \Omega \subset \R^2,$ from ray sinogram data. To be precise, the respective goal is to reconstruct solutions 
of the operator equation 
\begin{equation*}
\ray[w] = v, 
\end{equation*}
for $w \in W^{s,p}(\Omega;K)$ given some $v \in L^p(\Sigma;\R)$. 
Again we assume that only noisy data $\vd \in L^p(\Sigma;\R)$ are available, which in our tests was generated synthetically 
by adding Gaussian noise to the ideal data $\ray[w^\dagger]$. 

We thus minimize the regularization functional $\F_{[d]}^{\alpha,\vd}$ as defined in \autoref{eq:reg} over $W^{s,p}(\Omega;K)$ with $\d$ 
as defined as in \autoref{eq:vfMetrik} and $\op = \ray$. The regularization technique is well-defined as documented in \autoref{subsec:vfdata}.

We use $p=2$, different values of $r$ 
and different regularization parameters for the angle and the length which we denote by $\alpha$ and $\alpha \gamma$, see \autoref{eq:vfMetrik}. In 
the definition of $K = \mathcal{B}_r(0)\setminus \mathcal{B}^\circ_\varepsilon(0)$ we used as value for $\varepsilon$ the floating point relative accuracy $\mathtt{eps}$ in $\mathtt{Matlab}$ .

The {\bf first numerical test} concerns a vector field function 
$w^\dagger \in W^{s,p}(\Omega;\mathcal{B}_{0.1}(0) \setminus \mathcal{B}^\circ_\varepsilon(0) ) $, which is represented in \autoref{sfig:sfig3-c}. 
The length of the vectors of $w^\dagger$ vary between $0.01$ and $0.1$. The corresponding sinogram can be seen in \autoref{sfig:sfig3-a} 
to which we added Gaussian noise with variance $0.05$ to get our synthetic test data, see \autoref{sfig:sfig3-b}. 
In this example varying the angle has much more influence than varying the length, so it seems reasonable to regularize angle and length separately. 
The reconstructed vector fields using the metric double integral regularization with $d$ as defined in \autoref{eq:vfMetrik} confirm this, see \autoref{sfig:sfig3-d} and \autoref{sfig:sfig3-h}.

Using the Sobolev semi-norm regularization, \autoref{eq:VSN}, with a small parameter $\beta$ the result is much worse with respect to the 
$snr$-value, see \autoref{sfig:sfig3-f}. Increasing $\beta$ leads to a reduction of length near the jump of the angle, see \autoref{sfig:sfig3-j}.

The {\bf second numerical test} concerns reconstruction of a vector field function which has normalized length 1; it is shown in \autoref{sfig:sfig4-c}. 
We added Gaussian noise with variance $10$ to its sinogram, cf. \autoref{sfig:sfig4-a} and \autoref{sfig:sfig4-b}, 
and chose $r=1$.\\
Small regularization parameters lead to bad results in both cases since quite noisy sinogram data are given, see \autoref{sfig:sfig4-d} and \autoref{sfig:sfig4-f}.  
When increasing the parameters the advantage of using $\d$ as in \autoref{eq:vfMetrik} in contrast to the vectorial Sobolev semi-norm gets more obvious:   
Using our regularization method the resulting vectors stay close to the ideal data, see \autoref{sfig:sfig4-h}.
There are differences in the middle of the field which can be explained because of the separate minimization with respect to the angle. 
Increasing the regularization parameter and using $\F_{\text{VSN}}$ in \autoref{eq:VSN}
leads to a smoothing of the whole field, see \autoref{sfig:sfig4-j}. It is not possible to get a reconstructed field which preserves the jumps or length. This is due to the fact that the Sobolev semi-norm regularization does not decompose the vector into angle and length.

As a {\bf last example} we reconstruct the normalized vector field function 
seen in \autoref{sfig:sfig5-c} and use again $r=1$. Gaussian noise with variance $5$ is added to its sinogram, 
see \autoref{sfig:sfig5-a} and \autoref{sfig:sfig5-b}. 
Small regularization parameters lead to good results in both cases, see \autoref{sfig:sfig5-d} and \autoref{sfig:sfig5-f}, although noise is still visible. 
As in the second example, our regularization with a proper choice of parameters leads the resulting vectors to stay close to the ideal data, see \autoref{sfig:sfig5-h}, even in the area of the curl. 
This is not possible minimizing $\F_{\text{VSN}}$, see \autoref{sfig:sfig5-j}. Here, in contrast, we can observe significant length reduction near the curl.

\begin{figure}[!h]
	\centering
	\begin{subfigure}[h]{0.27\linewidth}
		\includegraphics[width=1\linewidth]{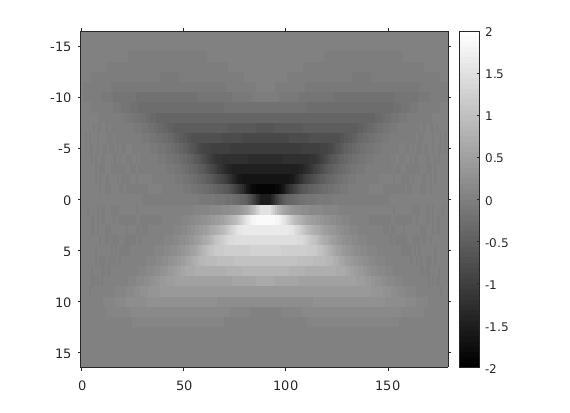}
		\caption{Sinogram of original data, $v^0$. }
		\label{sfig:sfig3-a}
	\end{subfigure}
	\begin{subfigure}[h]{0.27\linewidth}
		\includegraphics[width=1\linewidth]{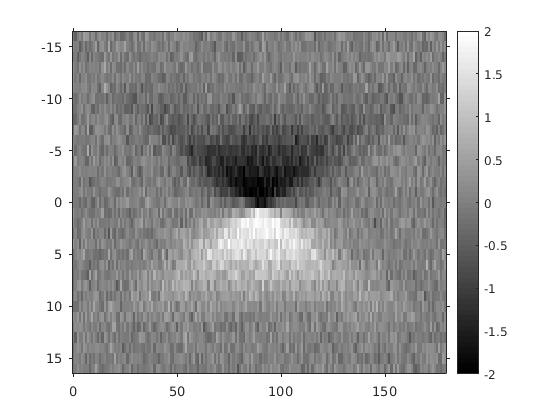}
		\caption{Noisy sinogram $\vd$, $\sigma^2 = 0.05$.}
		\label{sfig:sfig3-b}
	\end{subfigure}	
	\begin{subfigure}[h]{0.27\linewidth}
		\includegraphics[width=1\linewidth]{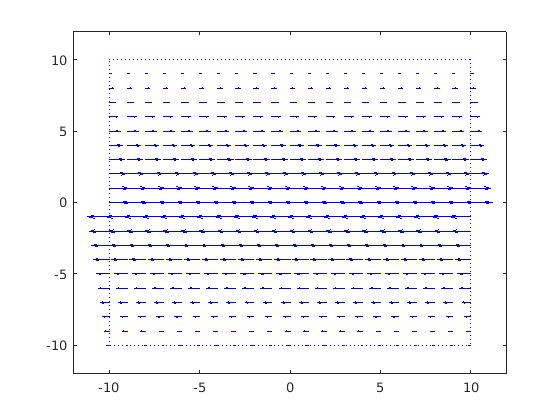}
		\caption{Original vector field $w^\dagger$. }
		\label{sfig:sfig3-c}
	\end{subfigure}		
	
	\begin{subfigure}[h]{0.27\linewidth}
		\includegraphics[width=1\linewidth]{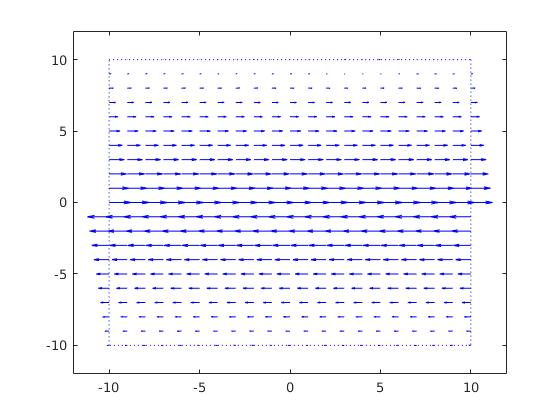}
		\caption{Result with metric double integral regularization, $\alpha = 0.01, \gamma = 1, \ snr = 35.61$.}
		\label{sfig:sfig3-d}
	\end{subfigure}
	\hspace{1.5cm}
	\begin{subfigure}[h]{0.27\linewidth}
		\includegraphics[width=1\linewidth]{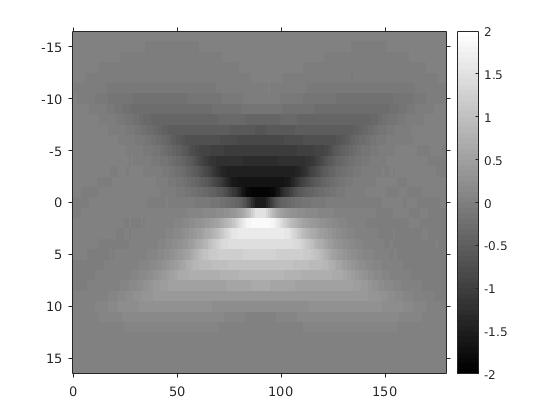}
		\caption{Sinogram using $\alpha = 0.01, \gamma = 1$.\newline }
		\label{sfig:sfig3-e}
	\end{subfigure}

	\begin{subfigure}[h]{0.27\linewidth}
		\includegraphics[width=1\linewidth]{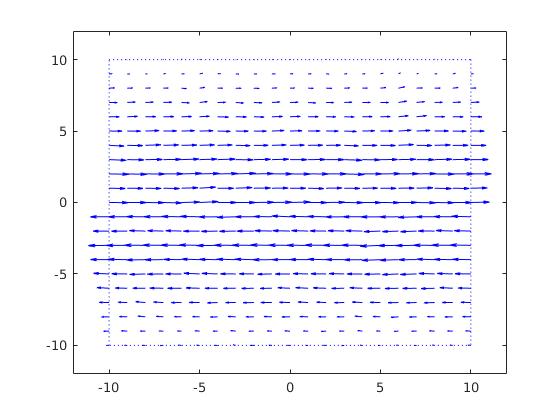}
		\caption{Reconstruction with the quadratic Sobolev semi-norm, $\beta = 0.01, \ snr = 14.42$.}
		\label{sfig:sfig3-f}
	\end{subfigure}
	\hspace{1.5cm}
	\begin{subfigure}[h]{0.27\linewidth}
		\includegraphics[width=1\linewidth]{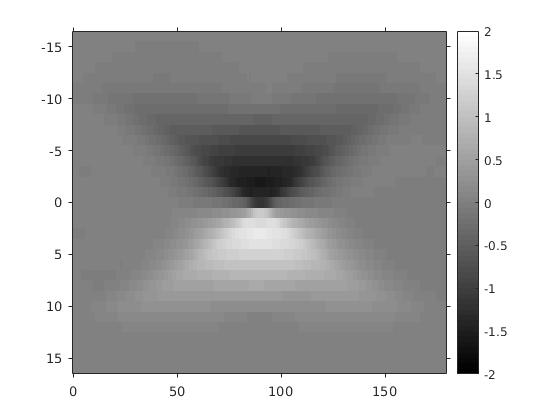}
		\caption{Sinogram using $\beta = 0.01$.\newline \quad \newline}
		\label{sfig:sfig3-g}
	\end{subfigure}

    \begin{subfigure}[h]{0.27\linewidth}
    	\includegraphics[width=1\linewidth]{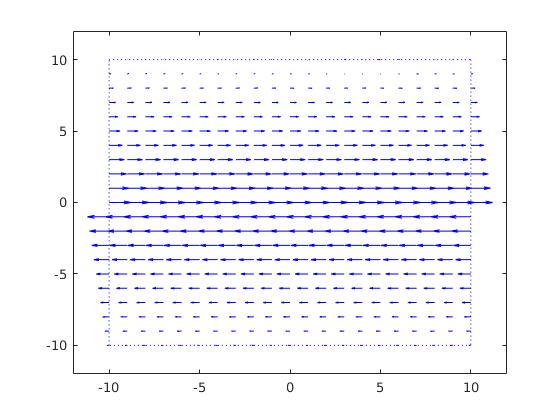}
    	\caption{Result with metric double integral regularization, $\alpha = 0.1, \gamma = 0.1, \ snr = 33.67$.}
    	\label{sfig:sfig3-h}
    \end{subfigure}
	\hspace{1.5cm}
    \begin{subfigure}[h]{0.27\linewidth}
    	\includegraphics[width=1\linewidth]{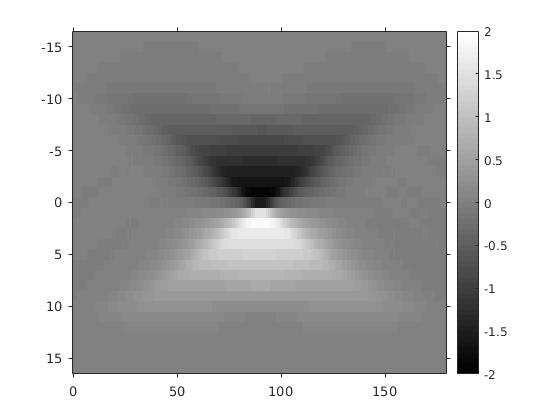}
    	\caption{Sinogram using $\alpha = 0.1, \gamma = 0.1$.\newline }
    	\label{sfig:sfig3-i}
    \end{subfigure}
	
	\begin{subfigure}[h]{0.27\linewidth}
		\includegraphics[width=1\linewidth]{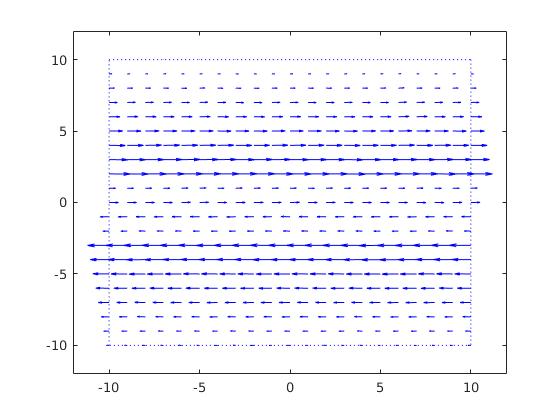}
		\caption{Reconstruction with the quadratic Sobolev semi-norm, $\beta = 0.1, \ snr = 6.51$.}
		\label{sfig:sfig3-j}
	\end{subfigure}
	\hspace{1.5cm}
	\begin{subfigure}[h]{0.27\linewidth}
		\includegraphics[width=1\linewidth]{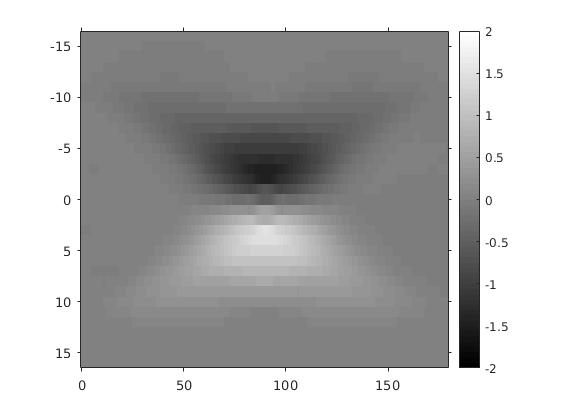}
		\caption{Sinogram using $\beta = 0.1$. \newline \quad \newline}
		\label{sfig:sfig3-k}
	\end{subfigure}
	\caption{\small Reconstruction of a vector field for fixed $p=2, s=0.49$ and $n_\rho=2$ with different regularization parameters $\alpha, \gamma$ and $\beta$.}
	\label{fig:fig3}
\end{figure}

\begin{figure}[!h]
	\centering
	\begin{subfigure}[h]{0.27\linewidth}
		\includegraphics[width=1\linewidth]{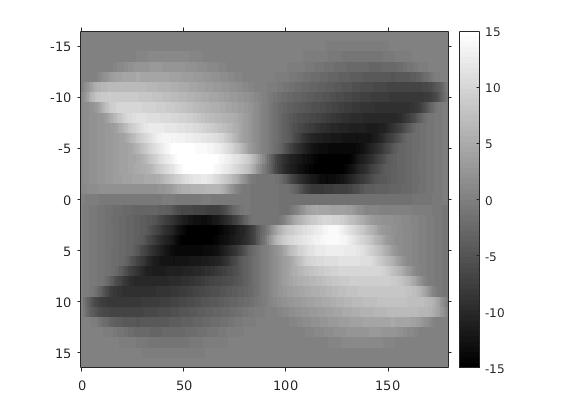}
		\caption{Sinogram of original data, $v^0$.}
		\label{sfig:sfig4-a}
	\end{subfigure}
	\begin{subfigure}[h]{0.27\linewidth}
		\includegraphics[width=1\linewidth]{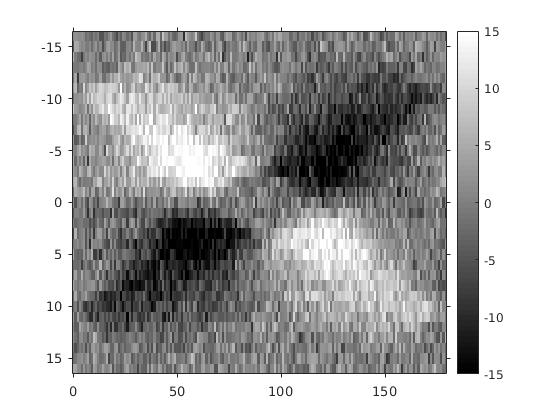}
		\caption{Noisy sinogram $\vd$, $\sigma^2 = 10$.}
		\label{sfig:sfig4-b}
	\end{subfigure}	
	\begin{subfigure}[h]{0.27\linewidth}
		\includegraphics[width=1\linewidth]{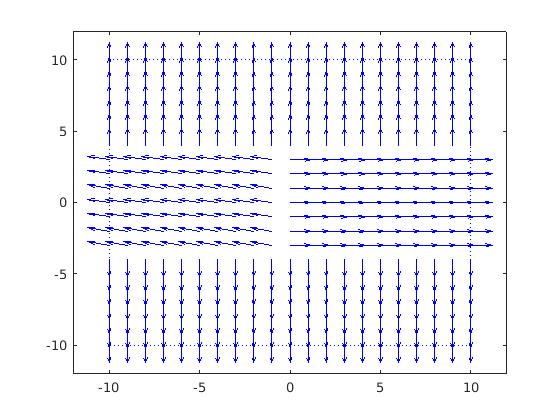}
		\caption{Original vector field $w^\dagger$. }
		\label{sfig:sfig4-c}
	\end{subfigure}		
	
	\begin{subfigure}[h]{0.27\linewidth}
		\includegraphics[width=1\linewidth]{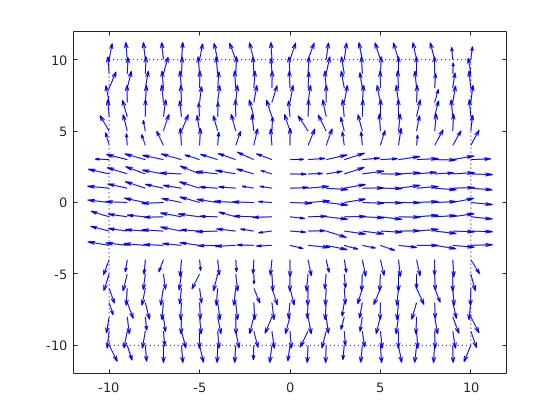}
		\caption{Result with metric double integral regularization, $\alpha = 0.001, \gamma = 1, \ snr = 14.96$.}
		\label{sfig:sfig4-d}
	\end{subfigure}
	\hspace{1.5cm}
	\begin{subfigure}[h]{0.27\linewidth}
		\includegraphics[width=1\linewidth]{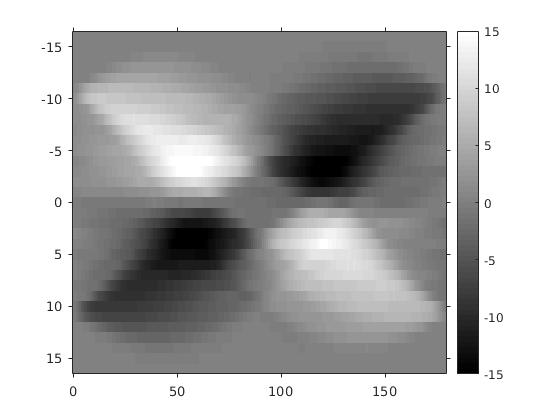}
		\caption{Sinogram using $\alpha = 0.001, \gamma = 1$.\newline }
		\label{sfig:sfig4-e}
	\end{subfigure}
	
	\begin{subfigure}[h]{0.27\linewidth}
		\includegraphics[width=1\linewidth]{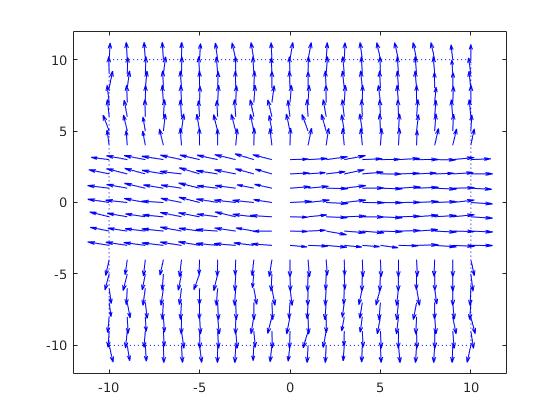}
		\caption{Reconstruction with the quadratic Sobolev semi-norm, $\beta = 0.001, \ snr = 19.83$.}
		\label{sfig:sfig4-f}
	\end{subfigure}
	\hspace{1.5cm}
	\begin{subfigure}[h]{0.27\linewidth}
		\includegraphics[width=1\linewidth]{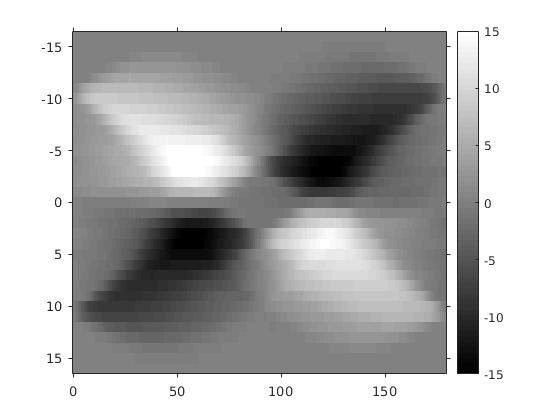}
		\caption{Sinogram using $\beta = 0.001$.\newline \quad \newline}
		\label{sfig:sfig4-g}
	\end{subfigure}
	
	\begin{subfigure}[h]{0.27\linewidth}
		\includegraphics[width=1\linewidth]{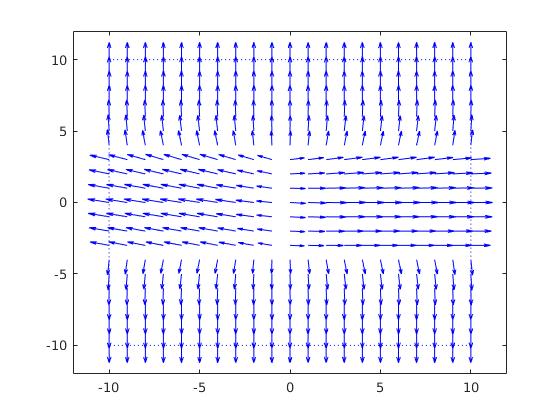}
		\caption{Result with metric double integral regularization, $\alpha = 0.1, \gamma = 4, \ snr = 18.08$.}
		\label{sfig:sfig4-h}
	\end{subfigure}
	\hspace{1.5cm}
	\begin{subfigure}[h]{0.27\linewidth}
		\includegraphics[width=1\linewidth]{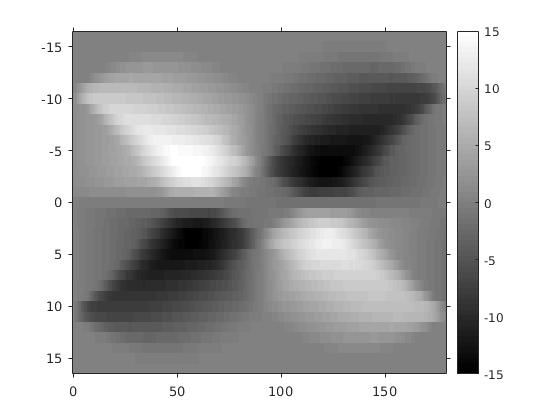}
		\caption{Sinogram using $\alpha = 0.1, \gamma = 4$. \newline \quad \newline}
		\label{sfig:sfig4-i}
	\end{subfigure}

	\begin{subfigure}[h]{0.27\linewidth}
		\includegraphics[width=1\linewidth]{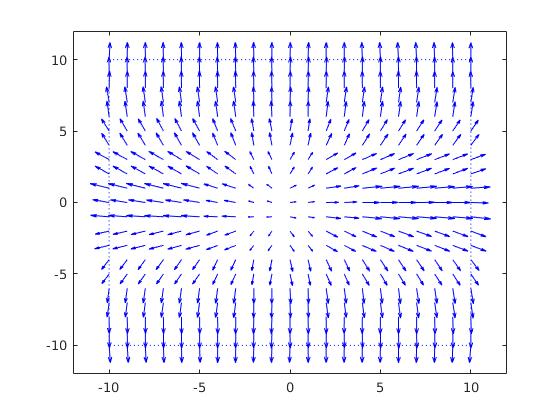}
		\caption{Reconstruction with the quadratic Sobolev semi-norm, $\beta = 0.1, \ snr = 7.72$.}
		\label{sfig:sfig4-j}
	\end{subfigure}
	\hspace{1.5cm}
	\begin{subfigure}[h]{0.27\linewidth}
		\includegraphics[width=1\linewidth]{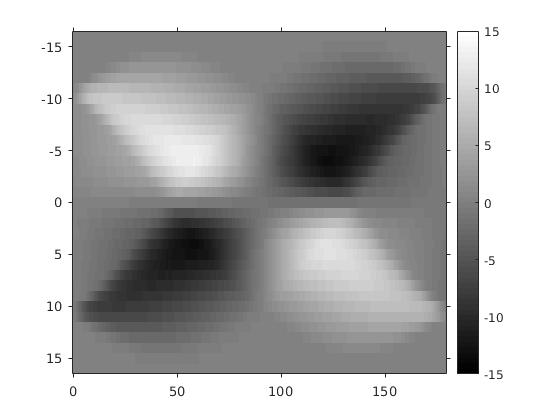}
		\caption{Sinogram using $\beta = 0.1$. \newline \quad \newline}
		\label{sfig:sfig4-k}
	\end{subfigure}   
 
	\caption{\small Reconstruction of a vector field for fixed $p=2, s=0.49$ and $n_\rho=3$ with different regularization parameters $\alpha, \gamma$ and $\beta$.}
	\label{fig:fig4}
\end{figure}

\begin{figure}[!h]
	\centering
	\begin{subfigure}[h]{0.27\linewidth}
		\includegraphics[width=1\linewidth]{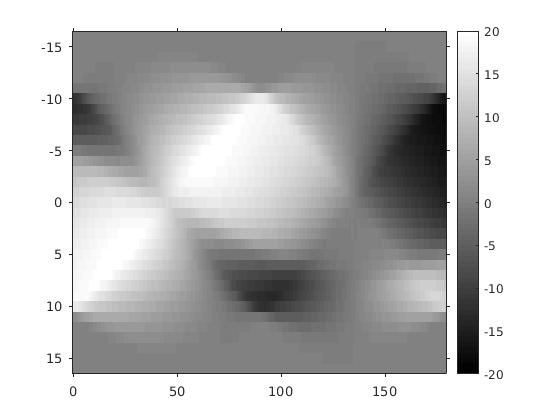}
		\caption{Sinogram of original data, $v^0$.}
		\label{sfig:sfig5-a}
	\end{subfigure}
	\begin{subfigure}[h]{0.27\linewidth}
		\includegraphics[width=1\linewidth]{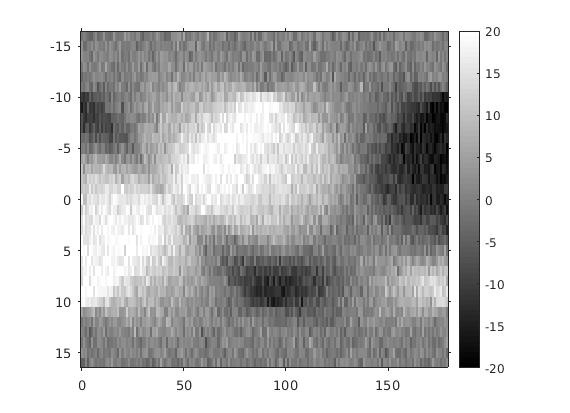}
		\caption{Noisy sinogram $\vd$, $\sigma^2 = 5$. }
		\label{sfig:sfig5-b}
	\end{subfigure}	
	\begin{subfigure}[h]{0.27\linewidth}
		\includegraphics[width=1\linewidth]{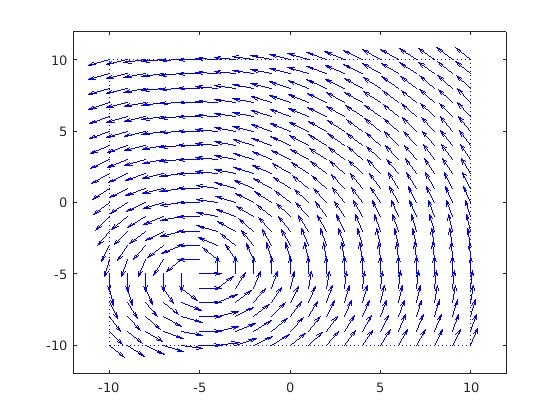}
		\caption{Original vector field $w^\dagger$. }
		\label{sfig:sfig5-c}
	\end{subfigure}		
	
	\begin{subfigure}[h]{0.27\linewidth}
		\includegraphics[width=1\linewidth]{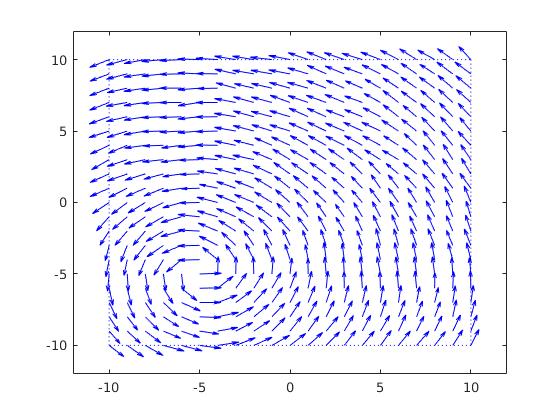}
		\caption{Result with metric double integral regularization, $\alpha = 0.001, \gamma = 1, \ snr = 20.71$.}
		\label{sfig:sfig5-d}
	\end{subfigure}
    \hspace{1.5cm}
	\begin{subfigure}[h]{0.27\linewidth}
		\includegraphics[width=1\linewidth]{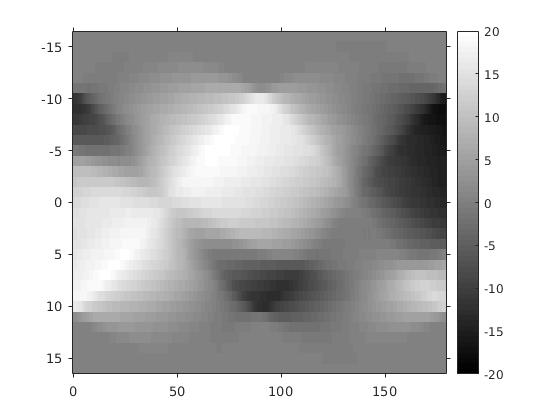}
		\caption{Sinogram using $\alpha = 0.001, \gamma = 1$. \newline}
		\label{sfig:sfig5-e}
	\end{subfigure}
	
	\begin{subfigure}[h]{0.27\linewidth}
		\includegraphics[width=1\linewidth]{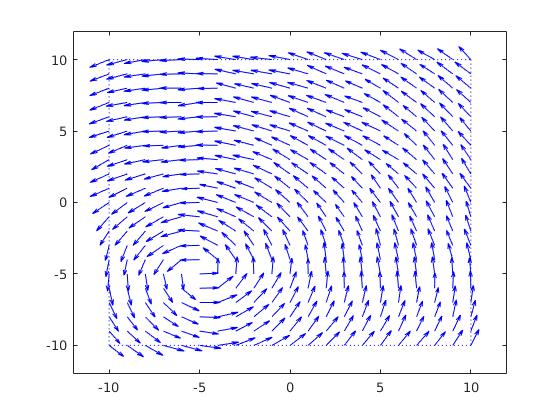}
		\caption{Reconstruction with the quadratic Sobolev semi-norm, $\beta = 0.001, \ snr = 20.44$.}
		\label{sfig:sfig5-f}
	\end{subfigure}
	\hspace{1.5cm}
	\begin{subfigure}[h]{0.27\linewidth}
		\includegraphics[width=1\linewidth]{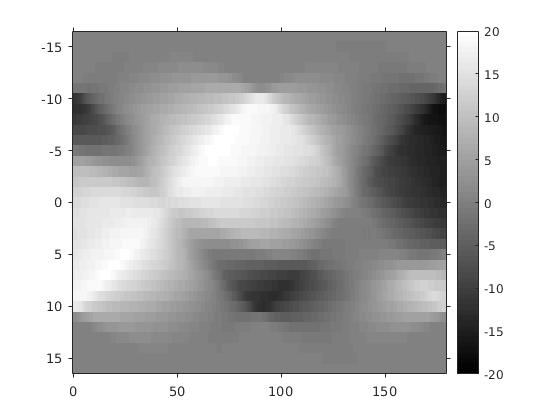}
		\caption{Sinogram using $\beta = 0.001$.\newline \quad \newline}
		\label{sfig:sfig5-g}
	\end{subfigure}
	
	\begin{subfigure}[h]{0.27\linewidth}
		\includegraphics[width=1\linewidth]{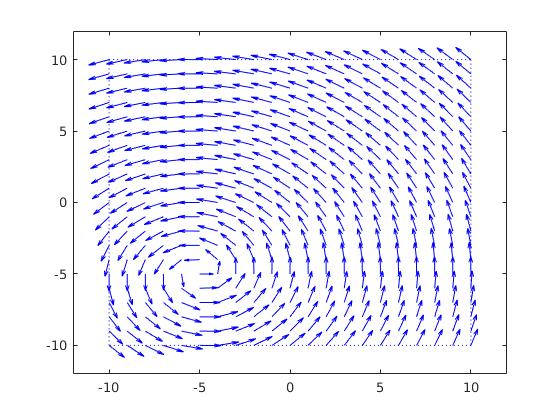}
		\caption{Result with metric double integral regularization, $\alpha = 0.1, \gamma = 3, \ snr = 27.57$.}
		\label{sfig:sfig5-h}
	\end{subfigure}
	\hspace{1.5cm}
	\begin{subfigure}[h]{0.27\linewidth}
		\includegraphics[width=1\linewidth]{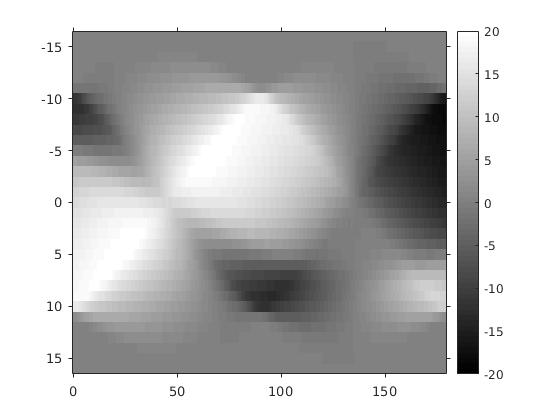}
		\caption{Sinogram using $\alpha = 0.1, \gamma = 3$.\newline \quad \newline}
		\label{sfig:sfig5-i}
	\end{subfigure}
	
	\begin{subfigure}[h]{0.27\linewidth}
		\includegraphics[width=1\linewidth]{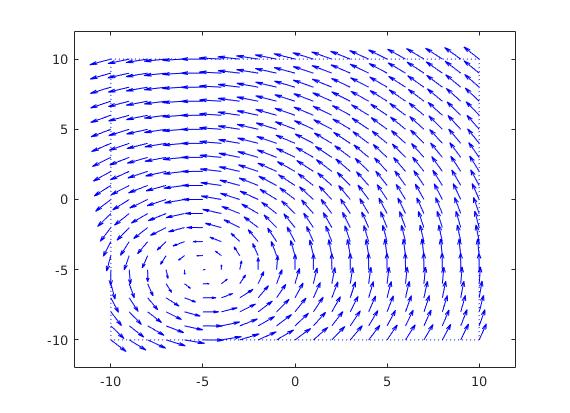}
		\caption{Reconstruction with the quadratic Sobolev semi-norm, $\beta = 0.1, \ snr = 15.91$.}
		\label{sfig:sfig5-j}
	\end{subfigure}
	\hspace{1.5cm}
	\begin{subfigure}[h]{0.27\linewidth}
		\includegraphics[width=1\linewidth]{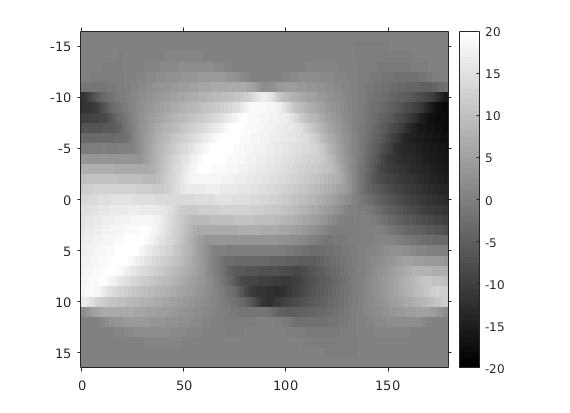}
		\caption{Sinogram using $\beta = 0.1$. \newline \quad \newline}
		\label{sfig:sfig5-k}
	\end{subfigure}    
	
	\caption{\small Reconstruction of a vector field for fixed $p=2, s=0.49$ and $n_\rho=1$ with different regularization parameters $\alpha, \gamma$ and $\beta$.}
	\label{fig:fig5}
\end{figure}

\clearpage

\section{Conclusion}
The contribution of this paper is the application of recently developed regularization methods for recovering 
functions with values in a closed set, typically an embedded sub-manifold, to vector tomographic imaging problems 
for the Radon and ray transform, respectively. These regularization methods have been investigated so far exclusively 
for image analysis problems, such as denoising and inpainting. 

\subsection*{Acknowledgements}
The authors acknowledge support from the Austrian Science Fund (FWF) under project I3661-N27 (Novel Error Measures and Source 
Conditions of Regularization Methods for Inverse Problems). Moreover, OS is supported by the Austrian Science Fund (FWF), 
with SFB F68, project F6807-N36 (Tomography with Uncertainties).


\section*{References}
\renewcommand{\i}{\ii}
\printbibliography[heading=none]

\end{document}